\documentclass[preprint, 12pt]{amsart}
\usepackage{fullpage,amsfonts,amsmath,amssymb, amsthm}
\usepackage{verbatim}

\def\frak{\mathfrak}
\theoremstyle{plain} \newtheorem{Thm}{Theorem}[section]
\theoremstyle{plain} \newtheorem{Cor}[Thm]{Corollary}
\theoremstyle{plain} \newtheorem{Prop}[Thm]{Proposition}
\theoremstyle{plain} \newtheorem{Lemma}[Thm]{Lemma}
\theoremstyle{definition} \newtheorem{Def}[Thm]{Definition}
\theoremstyle{definition} \newtheorem{Rem}[Thm]{Remark}
\theoremstyle{definition} 
\theoremstyle{definition} \newtheorem{Ex}[Thm]{Example}

\newcommand{\thmlist}{
\renewcommand{\theenumi}{\alph{enumi}}
\renewcommand{\labelenumi}{(\theenumi)}}
\renewcommand{\Re}{\mathop{\rm{Re}}}
\renewcommand{\Im}{\mathop{\rm{Im}}}

\newcommand{\id}{\mathop{\rm{id}}}

\newcommand{\inner}[2]{\langle#1,#2\rangle}

\newcommand{\C}{\ensuremath{\mathbb C}}

\newcommand{\D}{\ensuremath{\mathbb D}}

\newcommand{\R}{\ensuremath{\mathbb R}}
\newcommand{\Z}{\ensuremath{\mathbb Z}}
\newcommand{\N}{\ensuremath{\mathbb N}}

\renewcommand{\l}{\lambda}
\renewcommand{\a}{\alpha}
\renewcommand{\b}{\beta}

\newcommand{\fa}{\mathfrak{a}}

\newcommand{\la}{\l_\a}

\newcommand{\frakacs}{\frak a_{\C}^*}

\newcommand{\polya}{{\rm S}(\frak a_\C)}

\newcommand{\hyper}[4]{\ensuremath{\sideset{_{_2}}{_{_1}}{\mathop{F}}
\left(#1,#2;#3;#4\right)}}

\newcommand{\Hreg}{A_\C^{\rm reg}}

\renewcommand{\phi}{\varphi}

\begin{document}

\makeatletter
\title[Bounded hypergeometric functions]{Asymptotics of Harish-Chandra expansions,
bounded hypergeometric functions associated with root systems, and applications}
\author{E. K. Narayanan}
\address{Department of Mathematics, Indian Institute of Science, Bangalore -12, India}
\email{naru@math.iisc.ernet.in}
\author[UL]{A. Pasquale}
\address{
Laboratoire de Math\'ematiques et Applications de Metz (UMR CNRS 7122),
Universit\'e de Lorraine, F-57045 Metz, France\\
\textit{Present address:}
Institute Elie Cartan de Lorraine (IECL, UMR CNRS 7502),
Universit\'e de Lorraine, F-57045 Metz, France.}
\email{angela.pasquale@univ-lorraine.fr}
\author{S. Pusti}
\address{Mathematics Research Unit, University of Luxembourg,
Campus Kirchberg; 6, rue Richard Coudenhove-Kalergi;
L-1359, Luxembourg.}
\email{sanjoy.pusti@uni.lu}
\thanks{ The authors would like to thank Jean-Philippe Anker
for having sent them his article \cite{AnkerCRAS} and Patrick
Delorme for providing them with a preprint version of
\cite{Delorme}. This work was started when the first named author
was visiting the Universit\'e Paul Verlaine Metz. He thanks the
University for the invitation. He was also supported in part by a
grant from UGC center for Advanced Study. The second named author
would like to thank the Indian Institute of Science, Bangalore,
for
hospitality and financial support. 
}

\date{}
\subjclass[2010]{Primary: 33C67; secondary: 43A32, 43A90} 
\keywords{Hypergeometric functions, Harish-Chandra
series expansion, spherical functions, root systems, Cherednik
operators, hypergeometric Fourier transform}

\begin{abstract}
A series expansion for Heckman-Opdam hypergeometric functions
$\varphi_\l$ is obtained for all $\l \in \fa^*_{\mathbb C}.$ As a
consequence, estimates for $\varphi_\l$ away from the walls of a Weyl chamber are
established. We also characterize the bounded hypergeometric functions
and thus prove an analogue of the celebrated theorem of Helgason and Johnson on the bounded spherical functions on a Riemannian symmetric space of the noncompact type.
The $L^p$-theory for the hypergeometric Fourier transform is developed for $0<p<2$. In particular,
an inversion formula is proved when $1\leq p <2$.
\end{abstract}

\maketitle

\section*{Introduction}
A natural extension of Harish-Chandra's theory of spherical
functions on Riemannian symmetric spaces of the noncompact type
was introduced by Heckman and Opdam in the late eighties
(\cite{HOpd1}, \cite{Heck1}, \cite{Opd2}).
In this theory, the symmetric space is replaced by a triple
$(\mathfrak{a}, \Sigma, m)$ consisting of a finite dimensional
real Euclidean vector space $\mathfrak{a}$, a root system $\Sigma$
in the dual $\mathfrak a^*$ of $\mathfrak a$, and a positive multiplicity function
$m$ on $\Sigma$. A commuting family $\D=\D(\mathfrak a,\Sigma,m)$
of differential operators on $\mathfrak a$ is associated with this triple.
The hypergeometric functions of Heckman and Opdam are joint eigenfunctions of $\D$.
For certain values of the multiplicity function, the triple
$(\mathfrak{a}, \Sigma, m)$ indeed arises from a Riemannian symmetric
space of the noncompact type $G/K$. In this case, $\D$ coincides with
the algebra of radial components of the $G$-invariant differential operators
on $G/K$, and Heckman-Opdam's hypergeometric functions are the
restrictions to $\mathfrak a$ of Harish-Chandra's elementary spherical
functions on $G/K$.
Heckman-Opdam's theory of hypergeometric functions associated with
root systems underwent an important development with the
discovery of Cherednik operators (see \cite{Cherednik}, \cite{OpdamActa},
\cite{Opd00} and references therein). The Cherednik operators
(also called Dunkl-Cherednik
operators or trigonometric Dunkl operators, as they are the curved analogue
of the Dunkl operators on $\mathbb R^n$) are a commuting family
of first order differential-reflection operators. They allow to construct
algebraically all elements of $\D$.

Let $W$ denote the Weyl group of $\Sigma$, and let $C_c^\infty(\mathfrak{a})^W$
be the space of compactly-supported $W$-invariant smooth functions on $\mathfrak a$.
The spectral decomposition of $\D$ on $C_c^\infty(\mathfrak{a})^W$ is obtained
by means of the hypergeometric Fourier transform.
Let  $\mathfrak a_\C^*$ be the complexified dual of $\mathfrak a$, and let
$\varphi_\l$ denote the Heckman-Opdam's hypergeometric function of spectral parameter
$\l$. The hypergeometric Fourier transform $\widehat{f}$ (or $\mathcal F f$) of a
sufficiently regular $W$-invariant function $f$ on $\mathfrak a$ is defined by
integration against the $\varphi_\l$'s:
\begin{equation} \label{eq:hypFourier}
(\mathcal F f)(\l)=\widehat{f}(\l)=\int_{\mathfrak a} f(x) \varphi_{\l}(x) \, d\mu(x)\,.
\end{equation}
Here $d\mu(x)$ is a suitable measure on $\mathfrak a$ attached to the triple $(\mathfrak a, \Sigma,m)$
(see subsection \ref{subsection:hypergeomFourier}).
In \cite{OpdamActa}, Opdam established the basic results in the $L^2$-harmonic analysis
of the hypergeometric Fourier transform: the Paley-Wiener theorem, characterizing
the image under $\mathcal F$ of  $C_c^\infty(\mathfrak{a})^W$, the inversion formula on
$C_c^\infty(\mathfrak{a})^W$, and the Plancherel theorem.
The $L^2$-Schwartz space analysis was studied by Schapira \cite{Schapira}. See also \cite{Delorme}, where
the case of negative multiplicities has been considered.
On the other hand, to our knowledge, the $L^p$-harmonic analysis of the hypergeometric Fourier
transform has not yet been developed for $p\neq 2$.
The difficulty of extending the classical $L^p$-results for the spherical Fourier transform to
the context of the hypergeometric Fourier transform is related to the fact that several tools
coming from the geometry of the symmetric spaces are now missing. An example is Harish-Chandra's
integral formula for the spherical functions.

In this paper, our main result is Theorem \ref{thm:bddhyp}, which characterizes the hypergeometric
functions of Heckman and Opdam that are bounded. It is the necessary step for studying the
holomorphic properties of the hypergeometric Fourier transform of $L^1$ functions.
Our result is a natural extension of the celebrated theorem of Helgason and Johnson \cite{HeJo}
characterizing the bounded spherical functions on a Riemannian symmetric space $G/K$ of the noncompact type.
Suppose the triple $(\mathfrak a,\Sigma,m)$ arises from $G/K$.
Let $\rho \in \mathfrak a^*$ be defined by (\ref{eq:rhom}), and let $C(\rho)$ denote
the convex hull of the finite set $\{w\rho:w\in W\}$.
The theorem of Helgason and Johnson states that the spherical function $\varphi_\l$ on
$G/K$ is bounded if and only if $\l$ belongs to the tube domain in $\mathfrak a_\C^*$ over
$C(\rho)$. Theorem \ref{thm:bddhyp} proves that this characterization extends to
Heckman-Opdam's hypergeometric functions $\varphi_\l$ associated with any triple
$(\mathfrak a,\Sigma,m)$. Some partial results in this direction have also
been obtained by R\"osler in \cite{Ro}, by methods that are different from ours.

Our principal tools are appropriate series expansions of Heckman
and Opdam hypergeometric functions  $\varphi_\l$ on a positive
Weyl chamber $\mathfrak a^+$ of $\mathfrak a$.

For generic values of the spectral parameter $\l$, the function $\varphi_\l$ is defined on $\mathfrak a^+$ as linear combination of the Harish-Chandra series $\Phi_{w\l}$ with $w \in W$:
\begin{equation} \label{eq:varphiHCseries}
\varphi_\l=\sum_{w \in W} c(w\l) \Phi_{w\l}\,,
\end{equation}
where $c$ denotes Harish-Chandra's $c$-function; see (\ref{eq:c}) and (\ref{eq:FfromPhi}). By construction, the Harish-Chandra series are exponential series on $\mathfrak a^+$. So, for generic $\l$'s, one immediately obtains an exponential series expansion for $\varphi_\l$. But this expansion does
not extend to all spectral parameters: its coefficients are meromorphic functions of $\l \in \mathfrak{a}_\C^*$ and
have singularities for non-generic $\l$. To get series expansions which hold for all values of the spectral parameter, one needs a different method.

The idea behind the method used in this paper appeared in the study of spherical
functions on Riemannian symmetric spaces of the noncompact type.
In this case, the spherical function $\varphi_0$ can be recovered from
$\varphi_\l$ with $\l$ near $0$. Indeed, there is a polynomial
$\pi(\l)$ and a positive constant $C$ so that on $\mathfrak{a}$ we have
\begin{equation}
\label{eq:varphi0HC}
\varphi_0= C \partial(\pi) \big(\pi(\l) \varphi_\l)|_{\l=0},
\end{equation}
In (\ref{eq:varphi0HC}), $\partial(\pi)$ is the constant coefficient differential operator on $\mathfrak{a}$ canonically associated with $\pi$; see (\ref{eq:pi}) and section  \ref{subsection:hypsystem} for the precise definitions of $\pi$ and $\partial(\pi)$, respectively.
Formula (\ref{eq:varphi0HC}) originated in the work of Harish-Chandra (see e.g. \cite{GV}, p.
165 and references therein). It  was applied by Anker \cite{AnkerCRAS} to obtain an exponential series expansion and sharp estimates for Harish-Chandra's spherical functions $\varphi_0$.
In \cite{Schapira}, Schapira proved that the same
formula holds for the Heckman-Opdam's hypergeometric function
$\varphi_0$ and extended Anker's results to this case.
More precisely, the expansion of $\varphi_0$ comes from the explicit computation of the
right-hand side of  (\ref{eq:varphi0HC}) after substituting $\varphi_\l$ with
its exponential series expansion (\ref{eq:varphiHCseries}) for generic
$\l$'s. The point is that the polynomial $\pi(\l)$ cancels
all singularities near $\l=0$ of the meromorphic coefficients of the
expansion of $\varphi_\l$ (and it is the minimal polynomial having this property).

Our first step is Proposition \ref{prop:asymptoticsF-beginning}, where we prove an analog of (\ref{eq:varphi0HC}) for the Heckman-Opdam's hypergeometric functions of arbitrary spectral parameters. This requires a precise knowledge of the singularities of the coefficients of the exponential series expansion (\ref{eq:varphiHCseries}). For every fixed $\l_0 \in \mathfrak a_\C^*$ we find a
polynomial $p(\l)$ (depending on $\l_0$ and minimal in a suitable sense) so that multiplication by
$p(\l)$ cancels all singularities of the meromorphic coefficients of the exponential series expansion (\ref{eq:varphiHCseries}) in a neighborhood of $\l_0$. Notice that the analysis needed for $\l_0$
arbitrary is more delicate than the one for $\l_0=0$. Indeed, near $\l_0=0$,  the coefficients of the Harish-Chandra series are holomorphic. Consequently, only the singularities of Harish-Chandra's
$c$-function play a role. They are located along the root hyperplanes, hence along a Weyl-group invariant finite family of hyperplanes through $0$. On the other hand, near an arbitrary point $\l_0$,
one has to consider the singularities of the $c$-function, those of the coefficients of the
Harish-Chandra series, and one also needs to understand how they transform under the
action of the Weyl group. The polynomial $p(\l)$ is therefore a product of factors which take into account all these different contributions; see (\ref{eq:p}). The operator
$\partial(\pi)$ occurring in the analog of  (\ref{eq:varphi0HC}) for $\l_0$ is the constant coefficient differential operator associated with the highest order term $\pi(\l)$ of $p(\l)$. The general version
of  (\ref{eq:varphi0HC}),
\begin{equation}
\label{eq:HOspherical-diffop}
\varphi_{\l_0}=C \partial(\pi) \big(p(\l) \varphi_\l)|_{\l=\l_0}\,,
\end{equation}
is thus built using two polynomials, $p(\l)$ and $\pi(\l)$, which agree in the
very special case $\l_0=0$. The right-hand side of (\ref{eq:HOspherical-diffop}), together with the exponential expansion for generic $\l$'s, allows us to calculate the exponential series expansion of
$\varphi_{\l_0}$; see Theorem \ref{thm:HCseriesF} and Corollary \ref{cor:HCseriesF}.
 Some applications of the series expansions are then obtained in sections \ref{section:estimates} and \ref{section:bddhyp}.

Besides the characterization of the set of $\l$'s for which the
hypergeometric function $\varphi_\l$ is bounded, this article contains the following results in the asymptotic analysis of the Heckman-Opdam's hypergeometric functions $\varphi_\l$:
\begin{enumerate}
\item An exponential series expansion for $\varphi_\lambda$, even when
$\lambda \in \mathfrak a_\C^*$ is not generic (Theorem \ref{thm:HCseriesF}).
\item Estimates for $\varphi_\lambda$ away from the walls of the Weyl chambers for all
$\lambda \in \mathfrak{a}_{\mathbb C}^{*}$ (Theorem \ref{thm:leading-termF})
\item Sharp estimates for $\varphi_\lambda$ when $\lambda\in \mathfrak a^*$ (Theorem \ref{thm:real}).
\end{enumerate}

The sharp estimates of Theorem \ref{thm:real} were stated without proof in \cite[Remark 3.1]{Schapira}.

In the last section of the paper we develop the $L^p$-theory for
the hypergeometric Fourier transform for $0<p<2$. Using Theorem
\ref{thm:bddhyp}, we study the holomorphic properties of the
hypergeometric Fourier transform on $L^p(\fa,d\mu)^W$ when
$1\leq p<2$. We prove the Hausdorff-Young inequalities and the
Riemann-Lebesgue lemma. We also establish injectivity and an
inversion formula for the hypergeometric Fourier transform.  Then,
by an easy generalization of Anker's results \cite{AnkerJFA} for
the $L^p$-spherical Fourier transform on Riemannian symmetric
spaces, we prove an $L^p$-Schwartz space isomorphism theorem for
$0<p\leq 2$ (see Theorem \ref{Schwartz-space-iso}).

\section{Notation and preliminaries} \label{section:notation}

We shall use the standard notation $\N$, $\N_0$, $\Z$, $\R$ and
$\C$ for the positive integers, the nonnegative integers, the
integers, the reals and the complex numbers. The symbol $A\cup B$
denotes the union of $A$ and $B$, whereas $A \sqcup B$ indicates
their disjoint union. Given two nonnegative functions $f$ and $g$
on a domain $D$, we write $f \asymp g$ if there exists positive
constants $C_1$ and $C_2$ so that $C_1 g(x)\leq f(x) \leq C_2
g(x)$ for all $x \in D$.

Let $\frak a$ be an $l$-dimensional real Euclidean vector space
with inner product $\inner{\cdot}{\cdot}$, and let $\frak a^*$ be
the dual space of $\frak a$. For $\l\in \frak a^*$ let $x_\l \in
\frak a$ be determined by
$\l(x)=\inner{x}{x_\l}$ for all $x \in \frak a$. The assignment
$\inner{\l}{\mu}:=\inner{x_\l}{x_\mu}$ defines an inner product in
$\frak a^*$. Let $\frak a_\C$ and $\frak a_\C^*$
denote the complexifications of $\frak a$ and $\frak a^*$, respectively.
The $\C$-bilinear extension to $\frak a_\C$ and $\frak a_\C^*$ of the
inner products on $\frak a^*$ and $\frak a$ will also be denoted
by $\inner{\cdot}{\cdot}$. We shall often employ the notation
\begin{equation}\label{eq:la}
  \la:=\frac{\inner{\l}{\a}}{\inner{\a}{\a}}.
\end{equation}
We shall also set $|x|=\inner{x}{x}^{1/2}$ for $x \in \frak a$.

Let $\Sigma$ be a (possibly nonreduced) root system in $\frak a^*$
with associated Weyl group $W$. For $\a \in \Sigma$, we denote by
$r_\a$ the reflection $\l \mapsto \l-2 \la \a$ in $\frak a^*$. For
a set $\Sigma^+$ of positive roots in $\Sigma$, let
$\Pi=\{\a_1,\dots,\a_l\} \subset \Sigma^+$ denote the
corresponding set of simple roots. We denote by $\Sigma_0$ the
indivisible roots in $\Sigma$: if $\a \in \Sigma_0$, then $\a/2
\notin \Sigma$. We set $\Sigma_0^+=\Sigma^+ \cap \Sigma_0$.

A positive multiplicity function on $\Sigma$ is a $W$-invariant function
$m:\Sigma\rightarrow ]0,+\infty[$. Setting $m_\a:=m(\a)$ for $\a \in \Sigma$,
we therefore have $m_{w\a}=m_\a$ for all $w \in W$.
We extend $m$ to $\frak a^*$ by putting $m_\a=0$
for $\a \notin \Sigma$.
We say that a multiplicity function $m$ is \emph{geometric} if there
 is a Riemannian symmetric space of noncompact type $G/K$
with restricted root system $\Sigma$ such that $m_\a$ is the multiplicity
of the root $\a$ for all $\a\in \Sigma$. Otherwise, $m$ is said to be
\emph{non-geometric}.

The dimension $l$ of $\frak a$ will also be called the \emph{(real) rank} of
the triple $(\frak a, \Sigma,m)$.

In this paper we adopt the notation commonly used in the theory of symmetric spaces. It differs from the
notation in the work of Heckman and Opdam in the following ways. The root system $R$ and the multiplicity
function $k$ used by Heckman and Opdam
are related to our $\Sigma$ and $m$ by the relations $R=\{2\a:\a \in \Sigma\}$ and $k_{2\a}=m_\a/2$
for $\a \in \Sigma$.

We view $\frak a_\C$ of $\frak a$ as the Lie algebra of the
complex torus $A_\C:= \frak a_\C / \Z\{2\pi i x_\a/\inner{\a}{\a}:
\a \in \Sigma\}$. We write $\exp: \frak a_\C \rightarrow A_\C$ for
the exponential map, with multi-valued inverse $\log$. The split
real form $A:=\exp \frak a $ of $A_\C$ is an abelian subgroup with
Lie algebra $\frak a$ such that $\exp: \frak a \rightarrow A$ is a
diffeomorphism. In the following, to simplify the notation, we
shall identify $A$ with $\mathfrak a$ by means of this
diffeomorphism.

The action of $W$ extends to $\frak a$ by duality,
to $\frakacs$ and $\frak a_\C$ by $\C$-linearity. Moreover, $W$ acts on functions $f$ on any of these spaces
by $(wf)(x):=f(w^{-1}x)$, $w \in W$.

The positive Weyl chamber $\frak a^+$ consists of the elements  $x\in\frak a$ for which
$\a(x)>0$ for all $\a \in \Sigma^+$; its closure is
$\overline{\frak a^+}=\{x \in \frak a: \text{$\a(x) \geq 0$ for all $\a \in \Sigma^+$}\}$.
Dually, the positive Weyl chamber $(\frak a^*)^+$ consists of the elements  $\l\in\frak a^*$ for which
$\inner{\l}{\a}>0$ for all $\a \in \Sigma^+$. Its closure is denoted $\overline{(\frak a^*)^+}$.
We write $\l \leq \mu$ if $\l, \mu \in \mathfrak a^*$ and $\mu-\l \in \overline{(\frak a^*)^+}$.
The sets $\overline{\frak a^+}$ and $\overline{(\frak a^*)^+}$
are fundamental domains for the action of $W$ on $\frak a$ and $\frak a^*$, respectively.

The restricted weight lattice of $\Sigma$ is
  $P=\{\l \in \frak a^*:\la\in\Z \quad \text{for all $\a \in\Sigma$}\}.$
Observe that $\{2\a:\a\in \Sigma\} \subset P$.
If $\l \in P$, then the exponential $e^\l:A_\C\rightarrow \C$ given
by $e^\l(h):=e^{\l(\log h)}$ is single valued. The $e^\l$ are the algebraic characters
of $A_\C$. Their $\C$-linear span coincides with the ring
of regular functions $\C[A_\C]$ on the affine algebraic variety $A_\C$.
The lattice $P$ is $W$-invariant, and the Weyl group acts on $\C[A_\C]$
according to $w(e^\l):=e^{w\l}$.
The set $\Hreg:=\{h \in A_\C: e^{2\a(\log h)}\neq 1 \ \text{for all $\a \in \Sigma$}\}$
consists of the regular points of $A_\C$ for the action of $W$. Notice that $A^+ \equiv \mathfrak a^+$ is a
subset of $\Hreg$. The algebra $\C[\Hreg]$ of regular functions on $\Hreg$
is the subalgebra of the quotient field of $\C[A_\C]$ generated by
$\C[A_\C]$ and by $1/(1-e^{-2\a})$ for $\a \in \Sigma^+$.
Its $W$-invariant elements form the subalgebra $\C[\Hreg]^W$.

\subsection{Cherednik operators and the hypergeometric system} \label{subsection:hypsystem}
In this subsection we outline the theory of hypergeometric
differential equations associated with root systems. This theory
has been developed by Heckman, Opdam and Cherednik. We refer the
reader to \cite{Cherednik}, \cite{HeckBou}, \cite{HS},
\cite{OpdamActa}, \cite{Opd00} for more details and further
references.

Let $\polya$ denote the symmetric algebra over $\frak a_\C$ considered as the
space of polynomial functions on $\frakacs$, and let $\polya^W$ be the
subalgebra of $W$-invariant elements.
Every $p \in \polya$ defines a
constant-coefficient differential operators $\partial(p)$ on
$A_\C$ and on $\frak a_\C$
such that $\partial(x)=\partial_x$ is the directional derivative in the
direction of $x$ for all $x \in \frak a$.
The algebra of the differential operators $\partial(p)$
with $p \in \polya$ will also be indicated by $\polya$.
Let  $\D(\Hreg):=\C[\Hreg]\otimes \polya$
denote the algebra of differential operators on $A_\C$ with coefficients
in $\C[\Hreg]$.
The Weyl group $W$ acts on $\D(\Hreg)$ according to
\begin{equation*}
  w\big(\phi\otimes \partial(p)\big):=w\phi \otimes \partial(wp).
\end{equation*}
We write $\D(\Hreg)^W$ for the subspace of $W$-invariant elements.
The space
$\D(\Hreg) \otimes \C[W]$ can be endowed with the structure of an associative algebra
with respect to the product
\begin{equation*}
 (D_1 \otimes w_1)\cdot (D_2 \otimes w_2)=D_1w_1(D_2) \otimes w_1w_2,
\end{equation*}
where the action of $W$ on differential operators is defined by
$(wD)(wf):=w(Df)$ for every sufficiently differentiable function
$f$. It is also a left $\C[\Hreg]$-module. Considering $D \in
\D(\Hreg)$ as an element of $\D(\Hreg) \otimes \C[W]$, we shall
usually write $D$ instead of $D \otimes 1$. The elements of the
algebra $\D(\Hreg) \otimes \C[W]$ are called the
differential-reflection operators on $\Hreg$. The
differential-reflection operators act on functions $f$ on $\Hreg$
according to $(D\otimes w)f:=D(wf)$.

For $x \in \frak a$ the \emph{Cherednik operator\/} (or Dunkl-Cherednik operator\/)  $T_x\in \D(\Hreg) \otimes \C[W]$
is defined by
\begin{equation*}
  T_x:=
\partial_x-\rho(x)+\sum_{\a\in \Sigma^+} m_\a \a(x) (1-e^{-2\a})^{-1}
\otimes (1-r_\a)
\end{equation*}
where
\begin{equation}
  \label{eq:rhom}
  \rho:=\frac{1}{2} \sum_{\a \in \Sigma^+} m_\a \a  \in \frak a^*.
\end{equation}

The Cherednik operators can also be considered as operators acting
on smooth functions on $\frak a$. This is possible because, as can
be seen from the Taylor formula, the term $1-r_\a$ cancels the
apparent singularity arising from the denominator $1-e^{-2\a}$.
The Cherednik operators commute with each other;  cf.
\cite[Section 2]{OpdamActa}. Therefore the map  $x \mapsto T_x$
on $\frak a$ extends uniquely to an algebra homomorphism $p
\mapsto T_p$ of $\polya$ into $\D(\Hreg) \otimes \C[W]$.

Define a linear map $\Upsilon:\D(\Hreg)\otimes \C[W] \rightarrow \D(\Hreg)$ by
\begin{equation*}
  \Upsilon(\sum_j D_j \otimes w_j):=\sum_j D_j.
\end{equation*}
Then $\Upsilon(Q)f=Qf$ for all  $Q \in \D(\Hreg) \otimes \C[W]$ and all $W$-invariant $f$
on $\Hreg$.

For $p \in \polya$ we set
$D_p:=\Upsilon\big(T_p\big)$. If $p \in \polya^W$, then
$D_p\in \D(\Hreg)^W$; see \cite[Theorem 2.12(2)]{OpdamActa}.
The algebra
\begin{equation*}
\D=\D(\frak a, \Sigma,m):=\{D_p:p \in \polya^W\}
\end{equation*}
is a commutative subalgebra of $\D(\Hreg)^W$.
It is called the algebra of
hypergeometric differential operators associated with $(\frak a,\Sigma,m)$.
It is the analogue,
for arbitrary multiplicity functions, of the commutative
algebra of the radial components on $A=\exp \mathfrak a$ of the
invariant differential operators on a Riemannian symmetric space of
noncompact type.

A remarkable element of $\D$ corresponds to the polynomial $p_L \in \polya^W$ defined by
$p_L(\l):=\inner{\l}{\l}$ for $\l \in \frakacs$.
Then
\begin{equation*}
   D_{p_L} =L+ \inner{\rho}{\rho},
\end{equation*}
  where
  \begin{equation}
    \label{eq:Laplm}
  L:=L_{\mathfrak a}+\sum_{\a \in \Sigma^+} m_\a \,\coth\a \;
       \partial_\a
  \end{equation}
and $L_{\mathfrak a}$ is the Laplace operator on $\mathfrak a$; 
see \cite[Theorem 2.2]{HeckBou}. In (\ref{eq:Laplm}) we have set $\partial_\a:=\partial(x_\a)$
and
\begin{equation*}
   \coth \a:=\frac{1+e^{-2\a}}{1-e^{-2\a}}.
\end{equation*}
If $(\mathfrak a,\Sigma,m)$ is geometric, then $L$ coincides with the radial component on $\mathfrak a\equiv A$ with respect to the left action of $K$ of the Laplace operator on a Riemannian symmetric space $G/K$ of noncompact type.

The map $\gamma: \D\rightarrow \polya^W$ defined by
\begin{equation}
  \label{eq:HChomo}
\gamma\big(D_p\big)(\l):=p(\l)
\end{equation}
is called the Harish-Chandra homomorphism.
It defines an algebra isomorphism of $\D$ onto $\polya^W$ (see \cite[Theorem 1.3.12 and Remark 1.3.14]{HS}). 
From Chevalley's theorem it therefore follows that
$\D$ is generated by $l (=\dim \frak a)$ elements.

Let $\l \in \frakacs$ be fixed. The system of differential equations
\begin{equation}
  \label{eq:hypereq}
  D_p \varphi=p(\l)\varphi, \qquad p \in \polya^W,
\end{equation}
is called the hypergeometric system of differential equations with spectral parameter
$\l$ associated with the data $(\frak a,\Sigma,m)$. The differential equation corresponding
to the polynomial $p_L$ is
\begin{equation}
  \label{eq:diffLaplm}
  L \varphi = \big(\inner{\l}{\l}-\inner{\rho}{\rho}\big)\varphi.
\end{equation}
For geometric multiplicities, the hypergeometric system (\ref{eq:hypereq}) agrees
with the system of differential equations defining Harish-Chandra's spherical function
of spectral parameter $\l$.

\begin{Ex}[The rank-one case] \label{ex:rank-one-diff}
The rank-one case corresponds to triples $(\frak a, \Sigma, m)$ in which
$\frak a$ is one dimensional. Then the set $\Sigma^+$ consists
at most of two elements: $\a$ and, possibly, $2\a$.
By setting $x_\a/2\equiv 1$ and $\a\equiv 1$,
we identify $\frak a$ and $\frak a^*$ with $\R$, and their complexifications
$\mathfrak a_\C$ and $\frakacs$ with $\C$.
The Weyl chamber $\frak a^+$ coincides with the half-line $]0,+\infty[$.
The Weyl group $W$ reduces to $\{-1,1\}$ acting on $\R$ and $\C$ by multiplication.
The algebra $\D$ is generated by
$D_{p_L}=L+\rho^2$. The hypergeometric differential
system with spectral parameter $\l \in \C$ is equivalent to the single Jacobi differential equation
\begin{equation} \label{eq:hyperalge}
\frac{d^2 \varphi}{dz^2}+\big(m_\a\coth z+m_{2\a}\coth(2z)\big)\; \frac{d\varphi}{dz}
=(\l^2-\rho^2)\varphi.
\end{equation}
The function $z \mapsto e^z$ maps $\frak a_\C\equiv \C$ onto $A_\C\equiv \C^\times$.
Hence $\Hreg \equiv \C \setminus \{0,\pm 1\}$. The change of variable $\zeta:=(1-\cosh z)/2$
transforms (\ref{eq:hyperalge}) into the hypergeometric differential equation
\begin{equation*}
  \zeta(1-\zeta) \frac{d^2 \psi}{d\zeta^2}+[c-(1+a+b)\zeta] \frac{d\psi}{d\zeta} -ab\, \zeta=0
\end{equation*}
with parameters
\begin{equation*}
  a=\frac{\l+\rho}{2}, \qquad b=\frac{-\l+\rho}{2}, \qquad c=\frac{m_\a+m_{2\a}+1}{2}.
 \end{equation*}
\end{Ex}

\subsection{The Harish-Chandra series} \label{subsection:HCseries}

As in the classical theory of spherical functions on Riemannian
symmetric spaces, the explicit expression of the differential
equation (\ref{eq:diffLaplm}) suggested to Heckman and Opdam
\cite{HOpd1} to look for solutions on $\mathfrak a^+$ of the hypergeometric
system (\ref{eq:hypereq}) with spectral parameter $\l$ which are
of the form
\begin{equation*}
\Phi_\l(x)=e^{(\l-\rho) (x)}
\sum_{\mu \in \Lambda} \Gamma_\mu(\l) e^{-\mu(x)},
\qquad x \in \mathfrak a^+.
\end{equation*}
Here $\Lambda:=\left\{\sum_{j=1}^l n_j \a_j: n_j \in \N_0\right\}$
is the positive semigroup generated by the fundamental system of simple roots
$\Pi:=\{\a_1,\dots, \a_l\}$ in $\Sigma^+$.
For $\mu \in \Lambda \setminus \{0\}$, the coefficients
$\Gamma_\mu(\l)$ are rational functions of $\l \in \frakacs$
determined from the recursion relations
\begin{equation}\label{eq:recursionGammamu}
\inner{\mu}{\mu-2\l} \Gamma_\mu(\l) = 2 \sum_{\a\in \Sigma^+} m_\a
 \sum_{\substack{k\in \N\\\mu -2k\a \in \Lambda}}
 \Gamma_{\mu-2k\a}(\l)  \inner{\mu+\rho-2k\a-\l}{\a},
\end{equation}
with initial condition $\Gamma_0(\l)=1$. They are derived by
formally inserting the series for $\Phi$ into the differential
equation (\ref{eq:diffLaplm}). Let $\ell(\mu):=\sum_{j=1}^l n_j$
denote the level of $\mu=\sum_{j=1}^l n_j\a_j \in \Lambda$. It is
easy to check by induction on $\ell(\mu)$ that the recursion
relations imply $\Gamma_\mu(\l)=0$ unless $\mu=\sum_{j=1}^l n_j
\alpha_j$ with $n_j \geq 0$ and $n_j$ even for all $j=1,\dots,l$.
Hence the function $\Phi_\l(x)$ is in fact a sum over $2\Lambda$,
that is
\begin{equation}  \label{eq:HCseries}
\Phi_\l(x)=e^{(\l-\rho)(x)}
\sum_{\mu \in 2\Lambda} \Gamma_\mu(\l) e^{-\mu(x)},
\qquad x \in \mathfrak a^+.
\end{equation}
The function $\Phi_\l(x)$ is called the
\textit{Harish-Chandra series}.

Let $\mu \in 2\Lambda\setminus \{0\}$ be fixed. A priori, the
relation (\ref{eq:recursionGammamu}) uniquely define the rational
function $\Gamma_\mu(\l)$ on $\frak a_\C^*$ provided
$\inner{\tau}{\tau-2\l}\neq 0$ for all $\tau \in 2\Lambda\setminus
\{0\}$ with $\tau \leq \mu$. Opdam proved that, in fact, many of
these singularities are removable. Correspondingly, many of the
apparent singularities of the Harish-Chandra series are removable
as well.

In the following we adopt the notation
\begin{equation} \label{eq:hyperplane}
\mathcal H_{\a,r}=\{\l \in \frak a_\C^*: \l_\a=r\}\,.
\end{equation}
We shall consider meromorphic functions $f$ on $\mathfrak a_\C^*$
with singularities on a locally finite (generically infinite)
family $\mathbf{H}$ of affine complex hyperplanes $\mathcal
H_{\a,r}$. We say that $f$ has \emph{at most a simple pole along
$\mathcal H_{\a,r}$} if the function $\l\mapsto (\l_\a -r) f(\l)$
extends holomorphically to a neighborhood of $\mathcal H_{\a,r}
\setminus \bigcup_{\mathcal{H} \in \mathbf{H}, \mathcal{H}\neq
\mathcal H_{\a,r}} \mathcal{H}$.

\begin{Thm} \label{thm:polesGamma-Phi}
\begin{enumerate}
\thmlist
\item
Let $\mu \in 2\Lambda\setminus \{0\}$. Then the rational function $\Gamma_\mu(\l)$ has at most simple poles
located along the hyperplanes
$\mathcal H_{\a,n}$
with $\a \in \Sigma_0^+$, $n\in \N$ and $2n\a \leq \mu$.
\item
There is a tubular neighborhood $U^+$ of $A^+=\exp \mathfrak a^+$ in $A_\C$ so that the Harish-Chandra
series $\Phi_\l(x)$ is a meromorphic function of $(\l,x)\in \frak a_\C^* \times U^+$ with at most
simple poles along hyperplanes of the form $\mathcal H_{\a,n}$ with $\a \in \Sigma_0^+$ and $n\in \N$.
\end{enumerate}
\end{Thm}
\begin{proof}
This is Corollary 2.10 in \cite{Opd2}. See also \cite[Proposition 4.2.5]{HS} and \cite[Lemma 6.5]{Opd00}.
\end{proof}

\begin{Rem}
The neighborhood $U^+$ in Theorem \ref{thm:polesGamma-Phi} can be chosen of the form $A^+U_0$ where
$U_0$ is a connected and simply connected neighborhood of $e=\exp 0$ in $T=\exp(i\frak a)$
so that the function $\log$ is single valued on it.
Then all functions $e^{(\l-\rho)(\log h)}$
(with $\l \in \frakacs$ and $h \in A_\C$) are single valued and holomorphic on $A^+U_0$.
\end{Rem}

The convergence of the Harish-Chandra series can be studied by estimating its coefficients.
We record the following result, which is due to Opdam (\cite[Lemma 2.1]{Opd2}; see also \cite[Lemma 4.4.2]{HS}). It is an extension of the classical argument by Helgason in \cite{HelPW}, Lemma 4.1;
see also \cite[Ch. IV, Lemma 5.3]{He2}. We state it in a slightly modified form
(fixed multiplicity function and variable Weyl group element), which is more
suitable to our purposes.
The last part of the lemma is a consequence of the first part and Cauchy's integral
formula.

\begin{Lemma}  \label{lemma:conv-diffHCseries}
Let $U \subset \frak a_\C^*$ be an open set with compact closure $\overline{U}$, and let $w \in W$.
Let $d(\l)$ be a holomorphic function such that $d(\l)\Gamma_\mu(w\l)$ is holomorphic on $\overline{U}$ and all
$\mu \in 2\Lambda\setminus \{0\}$. Let $x_0 \in \frak a^+$ be fixed. Then there is a constant $M_{U,x_0}$ such that
$$|d(\l)\Gamma_\mu(w\lambda)|\leq M_{U,x_0} e^{\mu(x_0)}\,$$
for all $\mu \in 2\Lambda$ and $\l\in U$.
Hence the series
\begin{equation} \label{eq:modHCseries}
e^{(w\l-\rho)(x)} \sum_{\mu \in 2\Lambda} d(\l)\Gamma_\mu(w\l) e^{-\mu(x)}
\end{equation}
converges absolutely and uniformly in $(\l,x) \in U \times (x_0+ \overline{\frak a^+})$ to
$d(\l)\Phi_{w\l}(x)$.
Furthermore, for every $p\in \polya$ there is a constant $M_{p,U,x_0}$ such that
$$\Big|\partial(p)\Big(d(\l)\Gamma_\mu(w\lambda)\Big)\Big|\leq M_{p,U,x_0} e^{\mu(x_0)}\,$$
for all $\mu \in 2\Lambda$ and $\l\in U$. The series
(\ref{eq:modHCseries}) can therefore be differentiated
term-by-term and the differentiated series converges absolutely
and uniformly in $(\l,x) \in U \times (x_0+ \overline{\frak a^+})$
to $\partial(p)\big(d(\l)\Phi_{w\l}(x)\big)$.
\end{Lemma}

Notice that the explicit expression of the holomorphic function $d$ is not relevant.
We choose a specific function in Proposition \ref{prop:asymptoticsF-beginning}.

As in the Riemannian case, the Harish-Chandra series can be used to build a basis for the smooth
solutions on $\mathfrak a^+$ of the entire hypergeometric system with spectral parameter $\l$.
This is possible when $\l \in \frakacs$ is generic.

\begin{Def} \label{def:generic}
We say that $\l \in \frakacs$ is \emph{generic} if $\la \notin \Z$ for all $\a \in \Sigma_0$.
\end{Def}
Notice that, since $\l_\a=2\l_{2\a}$, the element $\l \in \frakacs$ is generic if and only if $\la \notin \Z$
for all $\a \in \Sigma$.
Moreover, since $(r_\b\l)_\a=\l_{r_\b\a}$ for all $\a, \b \in \Sigma$, the set of generic elements in $\frakacs$ is $W$-invariant.

\begin{Thm} \label{cor:thm}
Let $U^+$ be the tubular neighborhood of $A^+$ from Theorem \ref{thm:polesGamma-Phi}.
If $\l \in \frakacs$ is generic, then the set $\{\Phi_{w\l}(x): w\in W\}$ is a basis of the
solution space on $U^+$ of the hypergeometric
system (\ref{eq:hypereq}) with spectral parameter $\l$.
\end{Thm}
\begin{proof}
See \cite[Corollary 4.2.6]{HS} or \cite[Theorem 6.7]{Opd00}.
\end{proof}

\begin{Ex}[The rank-one case] \label{ex:Phirankone}
The solution of the Jacobi differential equation (\ref{eq:hyperalge}) on $(0,+\infty)$
that behaves asymptotically as $e^{(\l-\rho)t}$ for $t\rightarrow +\infty$
is
\begin{equation*}
  \label{eq:Phillrankone}
\Phi_{\l}(t)
=(2\sinh t)^{\l-\rho} \hyper{\frac{\rho-\l}{2}}{\frac{-m_\a/2+1-\l}{2}}{1-\l}{-\sinh^{-2}t},
\end{equation*}
where $\sideset{_{_2}}{_{_1}}{\mathop{F}}$ denotes the Gaussian hypergeometric function; see e.g. \cite[Chapter 2]{Er}.
The function $\Phi_{\l}(t)$ coincides with the Jacobi function of second kind
$\Phi^{(a,b)}_\nu(t)$ with parameters
$a=(m_\a+m_{2\a}-1)/2$, $b=(m_{2\a}-1)/2$ and $\nu=-i\l$
(see \cite[Section 2]{Koorn}).
\end{Ex}
\begin{Ex}[The complex case] \label{ex:complexPhi}
Let $m$ be geometric multiplicity of a root system $\Sigma$ which
is reduced (i.e. $2\a \notin \Sigma$ for every $\a \in \Sigma$).
If $m_\a=2$ for all $\a \in \Sigma$,  then $m$ corresponds to a
Riemannian symmetric space of the noncompact type $G/K$ with $G$
complex. The triple  $(\frak a, \Sigma, m)$ will be said to
correspond to a complex case.
In the complex case, we have
\begin{equation}
  \label{eq:Phicompl}
  \Phi_{\l}(x)=\Delta(x)^{-1}e^{\l(x)}.
\end{equation}
where
\begin{equation} \label{eq:Weylden}
\Delta:=\prod_{\a\in\Sigma^+}(e^\a-e^{-\a})
\end{equation}
is the Weyl denominator.
\end{Ex}

\subsection{The $\text{\itshape\lowercase{c}}$-function} \label{section:c}

For $\a \in \Sigma_0^+$ and $\l \in \frak a_\C^*$ we set
\begin{equation} \label{eq:calpha}
c_\a(\l)=
 \frac{2^{-\la} \; \Gamma(\la)}
{\Gamma\Big(\frac{\la}{2}+\frac{m_\a}{4}+\frac{1}{2}\Big)
\Gamma\Big(\frac{\la}{2}+\frac{m_\a}{4}+\frac{m_{2\a}}{2}\Big)}\,,
\end{equation}
where $\Gamma$ is the Euler gamma function. Harish-Chandra's
$c$-function is the meromorphic function on $\frak a_\C^*$ defined
by
\begin{equation}\label{eq:c}
c(\l)=c_{\text{\tiny HC}} \prod_{\a \in \Sigma_0^+} c_\a(\l)
\end{equation}
where $c_{\text{\tiny HC}}$ is a normalizing constant chosen so that $c(\rho)=1$.

Recall the notation $\mathcal H_{\a,r}:=\{\l \in \frakacs: \la=r\}$ from (\ref{eq:hyperplane}).
Observe that the equality $\mathcal H_{\a,n}=\mathcal H_{\b,m}$ with $\a,\b\in \Sigma_0^+$ and
$n,m \in \Z$ implies $\a=\b$ and $n=m$. From the singularities of the gamma function
we therefore obtain the following lemma.

\begin{Lemma}\label{lemma:poleszerosc}
The meromorphic function $c(\l)$ admits at most simple poles located along the hyperplanes
$$\mathcal H_{\a,-n}\qquad \text{with $\a \in \Sigma_0^+$ and $n\in\N_0$}\,.$$
The possible zeros of $c(\l)$ are located along the hyperplanes
\begin{align*}
&\mathcal H_{\a,-(m_\a/2+m_{2\a})-2n}\qquad \text{with  $\a \in \Sigma_0^+$ and $n\in\N_0$}\,,\\
&\mathcal H_{\a,-m_\a/2-1-2n}\qquad \text{with $\a \in \Sigma_0^+$ and $n\in\N_0$}\,.
\end{align*}
 \end{Lemma}

\subsection{The hypergeometric functions of Heckman and Opdam}
Let $\l\in\frak a^*_\C$. The hypergeometric function of spectral
parameter $\l$ is the unique analytic $W$-invariant function
$\varphi_\l(x)$ on $\mathfrak a$ which satisfies the system of
differential equations (\ref{eq:hypereq}) and which is normalized
by $\varphi_\l(0)= 1$. In the geometric case, with the identification of
$\mathfrak a$ with $A=\exp \mathfrak a$, the function $\varphi_\l$
agrees with the (elementary) spherical function of spectral
parameter $\l$. For $x \in \mathfrak a^+$ and generic $\l \in
\frak a^*_\C$, the hypergeometric function admits the
representation
\begin{equation} \label{eq:FfromPhi}
 \varphi_\l(x)=\sum_{w\in W}c(w\l) \Phi_{w\l}(x)\,.
\end{equation}

\begin{Ex}[The rank-one case] \label{ex:Phirankone-bis}
In the rank-one case, with the identifications introduced in
Example \ref{ex:rank-one-diff}, Heckman-Opdam's hypergeometric
function coincides with the Jacobi function of the first kind
\begin{equation*}
  \varphi_\l(t)=
\hyper{{\frac{m_\a/2+m_{2\a}+\l}{2}}}{{\frac{m_\a/2+m_{2\a}-\l}{2}}}{{\frac{m_\a+m_{2\a}+1}{2}}}
{-\sinh^2 t}.
\end{equation*}
\end{Ex}
\begin{Ex}[The complex case] \label{ex:complexPhi-bis}
In the complex case the multiplicity is geometric. The
hypergeometric functions of Heckman and Opdam agree with
Harish-Chandra's spherical functions. In this very special case,
they are given by the explicit formula
\begin{equation}
  \label{eq:Fcompl}
  \varphi_\l(x)=\frac{\pi(\rho)}{\pi(\l)}\; \frac{\sum_{w\in W} (\det w) e^{w\l(x)} }{\Delta(x)}.
\end{equation}
where
\begin{equation} \label{eq:pi}
\pi(\l)=\prod_{\a\in \Sigma_0^+} \inner{\l}{\a}
\end{equation}
and $\Delta$ is as in (\ref{eq:Weylden}). See e.g. \cite[p. 251]{GV}.
\end{Ex}

The nonsymmetric hypergeometric function of spectral parameter
$\l$ is the unique analytic function $G_\l(x)$ on $\mathfrak a$ which
satisfies the system of differential-reflection equations
\begin{equation} \label{eq:TG}
T_x G_\l=\l(x)G_\l\,, \qquad x \in \frak a\,,
\end{equation}
and which is normalized by $G_\l(0)= 1$. We have the relation
\begin{equation}
 \label{eq:FandG}
\varphi_\l(x)= \frac{1}{|W|} \sum_{w \in W} G_\l(wx)\,, \qquad x\in \mathfrak a\,.
\end{equation}
Schapira proved in \cite{Schapira} that $\varphi_\l$ is real and strictly positive for $\l\in\frak a^*$.
He also proved the fundamental estimate:
\begin{equation}
|\varphi_\l|\leq \varphi_{\Re\l}\,, \qquad \l \in \frak a_\C^*\,.
\end{equation}
We refer to \cite{OpdamActa}, \cite{HS} and \cite{Schapira}
for the proof of these statements and for further information.

\subsection{The hypergeometric Fourier transform}
\label{subsection:hypergeomFourier}

Let $dx$ denote a fixed normalization of the Haar measure on $\fa$. We associate with the triple $(\mathfrak a,\Sigma,m)$
the measure $d\mu(x)=\mu(x)\, dx$ on $\fa$, where
\begin{equation}
 \label{eq:mu}
\mu(x)=\prod_{\a \in\Sigma^+} \big|e^{\a(x)}-e^{-\a(x)}\big|^{m_\a}\,.
\end{equation}
Notice that when $(\fa, \Sigma, m)$ comes from a Riemannian symmetric space $G/K$,
then $d\mu$ is the component along $A\equiv\mathfrak a$ of the Haar measure on $G$ with respect to the
Cartan decomposition $G=KAK$.

Recall from (\ref{eq:hypFourier}) the definition of the hypergeometric Fourier transform
$\mathcal Ff=\widehat{f}$ of a sufficiently regular $W$-invariant functions on $\mathfrak a$.

Let $\Gamma$ be a $W$-invariant compact convex subset of $\fa$ and let
$q_\Gamma(\lambda)=\sup_{x\in\Gamma}\lambda(x)$, with $\l \in \fa$, be the supporting function of $\Gamma$.
The Paley-Wiener space $PW_\Gamma(\frakacs)^W$ consists of all $W$-invariant entire functions
$F$ on $\frakacs$ satisfying
\begin{equation}
\label{eq:PWGamma}
\sup_{\lambda\in\frakacs} (1+|\lambda|)^N e^{-q_\Gamma(\Re\lambda)}|F(\lambda)|<\infty
\end{equation}
for all $N\in \N_0$. We topologize $PW_\Gamma(\frakacs)^W$ by the
seminorms defined by the left-hand side of (\ref{eq:PWGamma}). Let
$C_\Gamma^\infty(\mathfrak a)^W$ denote the space of $W$-invariant
smooth functions on $\fa$ with support inside $\Gamma$. The space
$C_\Gamma^\infty(\mathfrak a)^W$ is considered with the topology
induced by $C^\infty_c(\fa)$. We will only be interested in two
situations: when $\Gamma$ is equal to $B_R:=\{x \in \fa:|x|\leq
R\}$ for some $R>0$, and when $\Gamma$ is equal to the polar set
$C_\Lambda$ of the convex hull of the set $\{w(\Lambda):w\in W\}$
(with $\Lambda \in \mathfrak a^*$). Recall that $C_\Lambda=\{x\in
\fa:\Lambda(x^+)\leq 1\}$, where $x^+$ is the unique element in
$\overline{\fa^+}$ of the Weyl group orbit of $x$. In the first
case, $q_{B_R}(\l)=R|\l|$. Hence, the usual
 Paley-Wiener space on $\mathfrak{a}_\C^*$ denoted by $PW(\frakacs)^W,$ is the union of
the spaces $PW_{B_R}(\frakacs)^W,$ with the inductive limit
topology.

The following theorem gives the basic results in the
$L^2$-harmonic analysis of the hypergeometric Fourier transform.
It is due to Opdam \cite{OpdamActa}, who proved the Paley-Wiener
theorem in a slightly less general form. In the classical
geometric case of Riemannian symmetric spaces of the noncompact
type and with $\Gamma=B_R$, the theorem is due to Helgason
\cite{HelPW} and Gangolli \cite{GangPW}; with $\Gamma=C_\Lambda$
it was proven by Anker \cite{AnkerJFA}. The fact that Anker's
results also extend to the non-geometric case was observed by
Schapira in \cite[p. 240]{Schapira}.

\begin{Thm} We keep the above notation and assumptions.
\begin{enumerate}
\thmlist
 \item \textup{(Paley-Wiener theorem)} \label{PWheorem}
The hypergeometric Fourier transform $\mathcal F$ is a topological isomorphism between
$C_c^\infty(\fa)^W$ and $PW(\frakacs)^W$. It restricts to a topological isomorphism between
$C_\Gamma^\infty(\fa)^W$ and $PW_\Gamma(\frakacs)^W$.
\item \textup{(Plancherel theorem)}
For a suitable normalization of the Haar measure $d\l$ on $i\fa^*$, the hypergeometric Fourier transform
$\mathcal F$ extends to an isometric isomorphism between $L^2(\fa,d\mu)^W$ and $L^2(i\fa^*, |c(\l)|^{-2}\,d\l)^W$.
\item \textup{(Inversion formula)} For $f \in C^\infty_c(\fa)^W$ we have
$$f(x)= \int_{i\fa^*} \widehat{f}(\l) \varphi_{-\l}(x) |c(\l)|^{-2}\,d\l$$
for all $x \in \fa$.
\end{enumerate}
\end{Thm}

\section{Harish-Chandra series expansion of the hypergeometric function}

In this section we study the Harish-Chandra expansion of the hypergeometric function $\varphi_\l(x)$
around arbitrary (i.e. not necessarily generic) $\l \in \frak a^*_\C$ under the assumption that
$x \in \frak a^+$ is sufficiently far from the walls of $\frak a^+$.
We begin with some properties of the centralizer of $\l$ that will be needed in the sequel.

We shall employ the notation $W_\l=\{w \in W: w\l=\l\}$ for the centralizer of
$\l \in \frak a_\C^*$ in $W$. For $\Theta \subset \Pi$ we denote by $W_\Theta$ the (standard parabolic) subgroup of $W$ generated by the reflections $r_\a$ with $\a \in \Theta$. Moreover, we write $\Sigma_\Theta$ for the subsystem of $\Sigma$ consisting of the roots
which can be written as linear combinations of elements from $\Theta$. Furthermore, we set
$\Sigma_{\Theta}^+=\Sigma_{\Theta} \cap \Sigma^+$ and $\Sigma_{\Theta,0}^+=\Sigma_{\Theta} \cap \Sigma^+_0$.

Recall the notation $\l_\a$ from (\ref{eq:la}).  For $\l \in \frak a_\C^*$ we define
\begin{align}
\Sigma_{\l}&=\{\a\in \Sigma_0^+: \l_\a \in \Z\}\,,\\
\Sigma_{\l}^>&=\{\a\in \Sigma_0^+: \l_\a \in \N\}\,,\\
\Sigma_{\l}^0&=\{\a\in \Sigma_0^+: \l_\a=0\}\,.
\end{align}

Suppose that $\Re\l\in \overline{(\frak a^*)^+}$.
Then $W_{\Re\l}=W_{\Theta(\l)}$ where $\Theta(\l)=\Sigma_{\Re\l}^0 \cap \Pi$. Observe that $W_{\Theta(\l)}$ is the Weyl group of $\Sigma_{\Theta(\l)}$ and $\Sigma_{\Theta(\l),0}^+=\Sigma_{\Re\l}^0$.

Let $w_\l\in W$ be chosen so that $w_\l\Im\l \in \overline{(\frak a^*)^+}$. Then there is $\Xi(\l) \subset \Pi$ so that $W_{\Im\l}=w_\l^{-1}W_{\Xi(\l)} w_\l=W_{w_\l^{-1} \Xi(\l)}$.
Moreover, $\Sigma_{w_\l^{-1} \Xi(\l),0}^+=\Sigma_{\Im\l}^0$.

\begin{Lemma} \label{lemma:elementary-prop-lambda0}
Let $\l \in \frak a_\C^*$ with $\Re\l\in \overline{(\frak a^*)^+}$, and keep the above notation.
Then
\begin{align}
&W_{\l}=W_{\Re\l} \cap W_{\Im\l}=W_{\Theta(\l)} \cap W_{w_\l^{-1} \Xi(\l)}\,,\\
&\Sigma_{\l}^0=\Sigma_{\Theta(\l),0}^+ \cap \Sigma_{w_\l^{-1} \Xi(\l),0}^+ \subset \Sigma_{\Theta(\l),0}^+\,,\\
&\Sigma_{\l}=\Sigma_{\l}^0 \sqcup \Sigma_{\l}^>\,,\\
&\Sigma_{\l}^> \cap \Sigma_{\Theta(\l),0}^+=\emptyset\,.
\label{eq:SigmainterSigmaTheta}
\end{align}
Moreover, let $w \in W_{\l}$. Then
\begin{align}
\label{eq:wSigmalambda00-uno}
&w(\Sigma_{\l}^0) \subset \Sigma_{\l}^0 \sqcup (-\Sigma_{\l}^0)\,, \\
\label{eq:wSigmalambda00-due}
&w(\Sigma_0^+ \setminus \Sigma_{\l}^0) = \Sigma_0^+ \setminus \Sigma_{\l}^0\,.
\end{align}
Furthermore, if $w \in W_{\Re\l}=W_{\Theta(\l)}$, then
\begin{align}
\label{eq:wSigmalambdacap-uno}
&w(-\Sigma_0^+) \cap \Sigma_{\l}^>=\emptyset\,,\\
\label{eq:wSigmalambdacap-due}
&w(\Sigma_0^+) \cap \Sigma_{\l}^>=\Sigma_{\l}^>\,.
\end{align}
\end{Lemma}
\begin{proof}
The first part of the Lemma is an immediate consequence of the discussion above and the fact that
$\Re\l\in \overline{(\frak a^*)^+}$. For (\ref{eq:SigmainterSigmaTheta}), observe that if $\a\in \Sigma_{\l}^>$ then
$\inner{\Re\l}{\a}=\inner{\l}{\a}>0$.
To prove (\ref{eq:wSigmalambda00-uno}), notice that
$w^{-1} \in W_{\l}$. Hence for $\a\in \Sigma_{\l}^0$ we have $\l_{w\a}=(w^{-1}\l)_\a=\l_\a=0$.
Formula (\ref{eq:wSigmalambda00-due}) is a consequence of (\ref{eq:wSigmalambda00-uno}) applied to $w$ and
$w^{-1}$ and the fact that $W_{\l}$ is the Weyl group of the closed subsystem $\Sigma_{\l}^0 \sqcup (-\Sigma_{\l}^0)$ of $\Sigma$.

Suppose now that $w \in W_{\Theta(\l)}$, so $w(\Re\l)=\Re\l$. Let $\a \in w(-\Sigma_0^+)$. Then
\begin{equation*}
\Re(\l_\a)=(\Re\l)_\a=(w(\Re\l))_\a=(\Re\l)_{w^{-1}\a} \leq 0
\end{equation*}
because $w^{-1}\a \in -\Sigma_0^+$ and $\Re\l \in \overline{(\frak a^*)+}$. Thus $\a\notin \Sigma_{\l}^>$.
This proves (\ref{eq:wSigmalambdacap-uno}). Finally, (\ref{eq:wSigmalambdacap-due}) follows immediately from
(\ref{eq:wSigmalambdacap-uno}).
\end{proof}

As intersection of parabolic subgroups of $W$, for each $\l\in \mathfrak a_\C^*$
the group $W_{\l}$ is itself a parabolic subgroup. By definition, this means that there is $I \subset \Pi$ and $w \in W$ so that $W_{\l}=W_{wI}=wW_Iw^{-1}$.  The proof of the following proposition provides an explicit way of constructing the elements $w$ and $I$ when $\l$ is given.

\begin{Prop}
Let $W_1, W_2 \subset W$ be two parabolic subgroups. Then $W_1 \cap W_2$ is parabolic.
\end{Prop}
\begin{proof} (see \cite[Proposition 3.11]{Qi}).
We can suppose that $W_1 \neq W_2$. Let $\l_1$ and $\l_2$ be two distinct elements of $\frak a^*$ fixed by $W_1$ and $W_2$, respectively. (Recall that every parabolic subgroup of $W$ is the centralizer of some element of $\mathfrak{a}^*$.)
Then $W_ 1 \cap W_2$ fixes the segment $\overline{\l_1 \l_2}$. Recall also that
$\frak a^* =\sqcup_ {w \in W, I \subset \Pi} (wC_I)$, where
$$C_I=\{\l \in \frak a^*: \text{$\inner{\a}{\l}=0$ for $\a \in I$,  $\inner{\a}{\l}>0$ for $\a \in \Pi \setminus I$}\}\,;$$
see e.g. \cite[Section 1.15]{Humpreys}.
We can therefore find $I \subset\Pi$, $w \in W$ and $\mu_1\neq \mu_2$ so that $\overline{\mu_1\mu_2} \subset \overline{\l_1\l_2} \cap wC_I$. Hence $\mu_1$
and $\mu_2$ admit the same centralizer $wW_{I}w^{-1}$. It follows that $wW_{I}w^{-1}$ fixes $\l_1$ and $\l_2$. Hence $wW_{I}w^{-1} \subset W_ 1\cap W_2$. Conversely, every element in $W_1 \cap W_2$ fixes $\mu_1, \mu_2 \in \overline{\l_1\l_2}$. So
$W_1 \cap W_2 \subset wW_{I}w^{-1}$.
\end{proof}

The following lemma describes the possible singularities of the coefficients of the Harish-Chandra expansion of $\varphi_{\l_0}(x)$ for an arbitrarily fixed $\l=\l_0 \in \mathfrak{a}_\C^*$. Notice that, by $W$-invariance, we can always suppose that $\Re\l_0\in \overline{(\frak a^*)^+}$. Some parts of this  lemma must have been considered by previous authors studying Harish-Chandra expansions. As we could not find references for them, we include their proof for the sake of completeness.

\begin{Lemma}
\label{lemma:singcGammamu}
Let $\l_0 \in \frak a_\C^*$ with $\Re\l_0\in \overline{(\frak a^*)^+}$, and keep the above notation.
Let $w \in W$ and $\mu \in 2\Lambda \setminus \{0\}$.
\begin{enumerate}
\thmlist \item The singularities of the function
$c(w\l)\Gamma_\mu(w\l)$ are at most simple poles along the
hyperplanes
$$\mathcal H_{\a,n}\qquad \text{with $\a\in \Sigma_0^+$ and $n \in \Z$}\,.$$
The hyperplane $\mathcal H_{\a,n}$ is a possible singular
hyperplane passing through $\l_0$ if and only if $\a\in
\Sigma_{\l_0}$ and $n=(\l_0)_\a$. \item The possible singularities
of $c(w\l)$ at $\l=\l_0$ are at most simple poles along the
hyperplanes
\begin{align*}
&\mathcal H_{\a,0}\qquad \text{with $\a\in \Sigma_{\l_0}^0$}\,,\\
&\mathcal H_{\a,n}\qquad \text{with $\a\in \Sigma_{\l_0}^> \cap
w(-\Sigma_0^+)$ and $n=(\l_0)_\a \in \N$}\,.
\end{align*}
In fact, each hyperplane $\mathcal H_{\a,0}$ with $\a\in
\Sigma_{\l_0}^0$ is always a simple pole of $c(w\l)$ at $\l=\l_0$.

The possible singularities of $\Gamma_\mu(w\l)$ at $\l=\l_0$ are
at most simple poles along the hyperplanes
\begin{equation*}
\mathcal H_{\a,n}\qquad \text{with $\a\in \Sigma_{\l_0}^> \cap
w(\Sigma_0^+)$ and $n=(\l_0)_\a \in \N$}\,.
\end{equation*}
\item Suppose $w \in W_{\Re\l_0}$. Then the singularities of
$c(w\l)$ at $\l=\l_0$ are precisely simple poles along the
hyperplanes
\begin{equation*}
\mathcal H_{\a,0}\qquad \text{with $\a\in \Sigma_{\l_0}^0$}\,.
\end{equation*}
Those of $\Gamma_\mu(w\l)$ at $\l=\l_0$ are at most simple poles along the hyperplanes
\begin{equation*}
\mathcal H_{\a,n} \qquad \text{with $\a\in \Sigma_{\l_0}^>$ and
$n=(\l_0)_\a \in \N$}\,.
\end{equation*}
\end{enumerate}
\end{Lemma}
\begin{proof}
Because of Theorem \ref{thm:polesGamma-Phi} and Lemma \ref{lemma:poleszerosc}, the possible singularities of
$c(w\l)\Gamma_\mu(w\l)$ are at most first order poles along hyperplanes of the form
$$\{\l \in \frak a_\C^*:w\l \in \mathcal H_{\a,n}\}=w^{-1}\mathcal H_{\a,n}$$
with $\a\in \Sigma_0^+$ and $n\in \Z$. Notice that $w^{-1}\mathcal H_{\a,n}=\mathcal H_{w^{-1}\a,n}$ and that
$\mathcal H_{-\b,-m}=\mathcal H_{\b,m}$.
The possible singular hyperplanes of $c(w\l)\Gamma_\mu(w\l)$ are hence of the form given in (a).
The last statement in (a) is immediate from the definition of $\mathcal H_{\a,n}$.

To prove (b), observe that $\mathcal H_{\a,n}$ is a possible singular hyperplane of $c(w\l)$ at $\l_0$
if and only if $\mathcal H_{w^{-1}\a,n}=w^{-1}\mathcal H_{\a,n}$ is a possible singular hyperplane of $c(\l)$ and
$\l_0 \in \mathcal H_{\a,n}$. The latter condition is equivalent to $\alpha \in \Sigma_{\l_0}=\Sigma_{\l_0}^0 \sqcup \Sigma_{\l_0}^>$
and $n=(\l_0)_\a$. If $\alpha \in \Sigma_{\l_0}^0$, then $n=0$ and $\mathcal H_{w^{-1}\a,0}$ is automatically a
singular hyperplane of $c(\l)$. If $\alpha \in \Sigma_{\l_0}^>$, then $n=(\l_0)_\a\in \N$.
In this case $\mathcal H_{w^{-1}\a,n}=\mathcal H_{-w^{-1}\a,-n}$ is a possible singular hyperplane of $c(\l)$ if and only if
$-w^{-1}\a \in \Sigma_0^+$.
Suppose now $\a \in \Sigma_{\l_0}^0$ and choose $\b \in \Sigma_0^+$ so that $w\a\in\{\pm \b\}$.
Because of the term $\Gamma\big((w\l)_\b\big)$ at the numerator of the factor $c_\b(w\l)$, the function $c(w\l)$
indeed admits a simple pole along $\{\l \in \mathfrak a_\C^*:(w\l)_\b=0\}=\{\l\in\frak a_\C^*: \l_{w^{-1}\b}=0\}=\mathcal H_{\a,0}$\,.

For $\Gamma_\mu(w\l)$, the hyperplane $\mathcal H_{\a,n}$ is a possible singular hyperplane at $\l_0$
if and only if $\mathcal H_{w^{-1}\a,n}=w^{-1}\mathcal H_{\a,n}$ is a possible singular hyperplane of $\Gamma_\mu(\l)$ and
$\l_0 \in \mathcal H_{\a,n}$. This is equivalent to saying that $\mathcal H_{w^{-1}\a,n}$ is a possible singular hyperplane
of $\Gamma_\mu(\l)$ and $n=(\l_0)_\a \in \N$, i.e. that $w^{-1}\a \in \Sigma_0^+$, $\a \in \Sigma_{\l_0}^>$ and
$n=(\l_0)_\a$.

Part (c) is a consequence of (b), (\ref{eq:wSigmalambdacap-uno}) and (\ref{eq:wSigmalambdacap-due}).
\end{proof}

To compute the exponential series expansion of the
hypergeometric function $\varphi_{\l_0}(x)$, we shall need some
elementary facts on polynomial differential operators. We collect
them in the following lemma whose proof is straightforward.

\begin{Lemma} \label{lemma:polydiff}
Let $\l_0\in \frak a_\C^*$, and let $I \subset \frak a_\C^* \times \frak a_\C^*$ be a finite set
so that $\inner{\l_0-\nu_2}{\nu_1}=0$ for all $(\nu_1,\nu_2) \in I$. Then the following properties
hold.
\begin{enumerate}
\thmlist
\item
Define polynomial functions $p_I$ and $\pi_I$ on $\frak a_\C^*$ by
\begin{align}
p_I(\l)&:=\prod_{(\nu_1,\nu_2) \in I} \inner{\l-\nu_2}{\nu_1}\,, \label{eq:pI} \\
\pi_I(\l)&:=\prod_{(\nu_1,\nu_2) \in I} \inner{\l}{\nu_1}\,. \label{eq:piI}
\end{align}
So $p_I(\l)=\pi_I(\l)+\widetilde{p_I}(\l)$ with $\deg \widetilde{p_I} < \deg p_I=\deg \pi_I=|I|$.
Then:
\smallskip

\begin{enumerate}
\renewcommand{\theenumii}{\arabic{enumii}}
\renewcommand{\labelenumii}{\theenumii)}
\item $\partial(\pi_I)(p)=0 \;$ if $p \in \polya$ and $\deg p< |I|$\,,
\smallskip

\item $\partial(p)(p_I)\big|_{\l=\l_0}=0 \;$ if $p \in \polya$ and $\deg p< |I|$\,,
\smallskip

\item
$\partial(\pi_I)(p_I)\big|_{\l=\l_0}=\partial(\pi_I)(\pi_I)>0$.
\smallskip
\end{enumerate}

\item
For every differentiable function $f$ on $\frak a_\C^*$ and every $x \in \frak a$ we have
\begin{equation}
\partial(\pi_I)\Big( f(\l)e^{w\l(x)}\Big)\Big|_{\l=\l_0}
= \sum_{J \sqcup L=I} \big(
\partial(\pi_J)f(\l)\big)\big|_{\l=\l_0} \; \pi_L(w^{-1}x)
e^{w\l_0(x)}\,
\end{equation}
where
\begin{equation}
\pi_L(x):=\prod_{(\nu_1,\nu_2) \in L} \nu_1(x)\,. \label{eq:pJa}
\end{equation}
\end{enumerate}
\end{Lemma}

We now contruct the polynomials $p(\l)$ and $\pi(\l)$ appearing on the right-hand side of
(\ref{eq:HOspherical-diffop}).

For $\alpha\in \Sigma_{\l_0}$ we set $n_\a=(\l_0)_\a \in \N_0$.
Let $w \in W$. Define polynomial functions $\pi_0(\l)$,
$\pi_{1}(\l)$, $p_{w,+}(\l)$, $p_{w,-}(\l)$, $p_1(\l)$ and $p(\l)
\in \polya$ by
\begin{align}
\pi_0(\l)&=\prod_{\a \in \Sigma_{\l_0}^0} \inner{\l}{\a}\,, \label{eq:pi0}\\
\pi_1(\l)&=\prod_{\a \in \Sigma_{\l_0}^>} \inner{ \l}{\a}\,,
\label{eq:pi1} \\
\pi(\l) & = \pi_0(\l) \pi_{1}(\l)=\prod_{\a \in \Sigma_{\l_0}} \inner{\l}{\a}\,, \label{eq:pilambda0}\\
p_{w,+}(\l)&=\prod_{\a \in \Sigma_{\l_0}^> \cap w(\Sigma_0^+)} (\inner{\l}{\a}-n_\a\inner{\a}{\a})\,,
\label{eq:pwplus}\\
p_{w,-}(\l)&=\prod_{\a \in \Sigma_{\l_0}^> \cap w(-\Sigma_0^+)} (\inner{\l}{\a}-n_\a\inner{\a}{\a})\,,
\label{eq:pwminus}\\
p_1(\l)&=p_{w,+}(\l)p_{w,-}(\l)\,, \label{eq:p1}\\
p(\l)&=\pi_0(\l)p_1(\l)\,.
\label{eq:p}
\end{align}
We adopt the convention that empty products are equal to the constant $1$.
Notice that $p_{w,+}(\l)p_{w,-}(\l)$ is in fact independent of $w \in W$ and that
\begin{equation}
\label{eq:setI} p(\l)=p_I(\l) \quad\text{and}\quad
\pi(\l)=\pi_I(\l)  \quad\text{for}\quad I=\{(\a,n_\a \a): \a\in
\Sigma_{\l_0}, n_\a=(\l_0)_\a\}\,.
\end{equation}

\begin{Prop}\label{prop:asymptoticsF-beginning}
Keep the assumptions of Lemma \ref{lemma:singcGammamu}.
\begin{enumerate}
\thmlist \item There is a neighborhood $U$ of $\l_0$ with compact
closure $\overline{U}$ so that $\overline{U}\cap \mathcal
H_{\a,n}\neq \emptyset$ if and only if $\a \in \Sigma_{\l_0}$ and
$n=(\l_0)_\a$. \item For all $w \in W$ and $\mu\in
2\Lambda\setminus \{0\}$, the functions
$\pi_0(\l)p_{w,-}(\l)c(w\l)$ and $p_{w,+}(\l)\Gamma_\mu(w\l)$ are
holomorphic in a neighborhood of $\overline{U}$. \item For all
$x\in A_\C$ where $\varphi_{\l_0}(x)$ is defined, we have
\begin{equation} \label{eq:F-derivative}
c_0 \varphi_{\l_0}(x)=\partial(\pi)\Big( p(\l)\varphi_\l(x)\Big)
\Big|_{\l=\l_0}
\end{equation}
where $c_0=\partial(\pi)(p)=\partial(\pi)(\pi)>0$. \item Let $x_0
\in \frak a^+$ be fixed. Then
\begin{equation} \label{eq:F-der-series}
c_0 \varphi_{\l_0}(x)=\sum_{\mu \in 2\Lambda} \sum_{w \in W}
\partial(\pi)\Big( p(\l)c(w\l)\Gamma_\mu(w\l)e^{(w\l-\rho-\mu)(x)}\Big) \Big|_{\l=\l_0}
\end{equation}
where the series on the right-hand side converges uniformly in
$x\in x_0+\overline{\frak a^+}$.
\end{enumerate}
\end{Prop}
\begin{proof}
Parts (a) and (b) are immediate consequences of Lemma \ref{lemma:singcGammamu} and the fact that the hyperplanes
$\mathcal H_{\b,n}$ ($\b\in \Sigma_0^+, n \in \Z$) form a locally finite family.

To prove (c), notice first that by parts 1) and 2) of Lemma
\ref{lemma:polydiff} (a),  we have for $J \subset I$
$$\partial(\pi_J)(p_I)\Big|_{\l=\l_0}=\delta_{J,I} \partial(\pi_I)(\pi_I)\,,$$
where $\delta_{J,I}$ is Kronecker's delta. Hence parts (b) and (c)
of the same lemma give for $I$ as in (\ref{eq:setI}):
\begin{align*}
\partial(\pi)\Big( p(\l)\varphi_\l(x)\Big)\Big|_{\l=\l_0}
&=\sum_{J\sqcup L=I} \Big(\partial(\pi_J)(p)(\l)\Big)\Big|_{\l=\l_0} \Big(\partial(\pi_L)\varphi_\l(x)\Big)\Big|_{\l=\l_0}\\
&=\partial(\pi)(\pi) \varphi_{\l_0}(x)\,.
\end{align*}

According to (b) and Lemma \ref{lemma:conv-diffHCseries},  the series
$$\sum_{\mu \in 2\Lambda} p_{w,+}(\l)\Gamma_\mu(w\l) e^{(w\l-\rho-\mu)(x)}$$
converges to $p_{w,+}(\l)\Phi_{w\l}(x)$ uniformly in $\overline{U} \times (x_0 + \overline{\frak a^+})$.
Moreover, it can be differentiated term-by-term.
If $\l$ is regular, then $\varphi_\l(x)=\sum_{w \in W} c(w\l) \Phi_{w\l}(x)$ for $x \in \mathfrak a^+$.
Multiplying both sides by $p(\l)$, we therefore get for all regular $\l\in U$
\begin{equation} \label{eq:pF}
p(\l)\varphi_\l(x)=\sum_{w \in W} \pi_0(\l)p_{w,-}(\l) c(w\l) \sum_{\mu \in 2\Lambda}
p_{w,+}(\l)\Gamma_\mu(w\l) e^{(w\l-\rho-\mu)(x)}\,.
\end{equation}
Since both sides of (\ref{eq:pF}) are holomorphic on $U$, this equality extends to all of $U$.
Part (d) now follows from (c) and term-by-term differentiation using Lemma \ref{lemma:conv-diffHCseries}.
\end{proof}

A careful computation of  (\ref{eq:F-der-series}) is what leads to
an expansion of $\varphi_{\l}(x)$ for all $\l.$  We first consider the
terms corresponding to $w \in W_{\l_0}$.

\begin{Lemma} \label{lemma:b0}
Let the notation be as above. Let $\l_0\in \mathfrak{a}_\C^*$ be such that
$\Re\l_0 \in \overline{(\frak a^*)^+}$. Define
\begin{equation}
\label{eq:b0}
b_0(\l)=\pi_0(\l) c(\l)\,.
\end{equation}
Then the following properties hold for $w \in W_{\l_0}$.
\begin{enumerate}
\thmlist
 \item $\pi_0(w\l)=(\det w) \pi_0(\l)$ for all $\l \in \frak a_\C^*$.
 \item The point $\l=\l_0$ is neither a zero nor a pole of the function $b_0(w\l)$.
 \item $b_0(w\l_0)=b_0(\l_0)$ is a nonzero constant. Moreover, if $\l_0 \in \overline{(\frak a^*)^+}$, then $b_0(\l_0)>0$.
 \end{enumerate}
 \end{Lemma}
\begin{proof}
Part (a) is classical, as $W_{\l_0}$ is the Weyl group of $\Sigma_{\l_0}^0 \sqcup (-\Sigma_{\l_0}^0)$.
 All other properties are an immediate consequence of the defining formula (\ref{eq:c})
 of the $c$-function, the properties of the gamma function, the fact that $w\l_0=\l_0$ and
 the assumption $\Re\l_0 \in \overline{(\frak a^*)^+}$.
\end{proof}

According to (c) in Lemma \ref{lemma:singcGammamu}, multiplication by $\pi_0$ removes all the singularities of $c(w\l)$ at
$\l_0$ for all $w \in W_{\l_0}$. To compute the corresponding terms in (\ref{eq:F-der-series}), we first rewrite them in
terms of the function $b_0$. Indeed, for $w \in W_{\l_0}$ we have by Lemma \ref{lemma:b0} (a):
\begin{equation} \label{eq:c-and-b}
p(\l)c(w\l)=\pi_0(\l)c(w\l)p_1(\l)=(\det w) b_0(w\l)p_1(\l)\,.
\end{equation}
The terms in (\ref{eq:F-der-series}) which correspond to $w \in W_{\l_0}$ and $\mu=0$ are then given by the following lemma.
Recall the set $I$ introduced in (\ref{eq:setI}) and recall that $\Gamma_0=1$.

\begin{Lemma} \label{lemma:termWl0-0}
Let $w \in W_{\l_0}$. Then for $x \in \frak a$ we have
\begin{multline}
\partial(\pi) \Big( b_0(w\l) p_1(\l) e^{(w\l-\rho)(x)}\Big)\Big|_{\l=\l_0}=\\
=\Big\{b_0(\l_0) \Big[ \sum_{\substack{J\sqcup L=I\\
|J|=|\Sigma_{\l_0}^>|}} \partial(\pi_J)(p_1)\big|_{\l=\l_0}
\pi_{wL}(x)\Big] +f_{w,\l_0}(x)\Big\} e^{(\l_0-\rho)(x)}
\end{multline}
where $f_{w,\l_0}(x)$ denotes a polynomial function of $x$ of degree $< \deg \pi_0=|\Sigma_{\l_0}^0|$.
\end{Lemma}
\begin{proof}
By Lemma \ref{lemma:polydiff}, we have
$\partial(\pi_J)(p_1)\big|_{\l=\l_0}=0$ for all $J \subset I$ with
$|J|< \deg p_1=|\Sigma_{\l_0}^>|$. Hence
\begin{multline*}
 \partial(\pi) \Big( b_0(w\l) p_1(\l) e^{(w\l-\rho)(x)}\Big)\Big|_{\l=\l_0}\\
= \sum_{\substack{J\sqcup L=I\\ |J|\geq |\Sigma_{\l_0}^>|}}
\partial(\pi_J)(p_1)\big|_{\l=\l_0}
 \partial(\pi_L)\Big(b_0(w\l)e^{(w\l-\rho)(x)}\Big)\Big|_{\l=\l_0}\,.
\end{multline*}
Notice that $\pi_L(w^{-1}x)=\pi_{wL}(x)$. Applying again Lemma
\ref{lemma:polydiff} (b), we get
\begin{align*}
&\partial(\pi_L)\Big(b_0(w\l)e^{(w\l-\rho)(x)}\Big)\Big|_{\l=\l_0}=\\
&\qquad =\sum_{R\sqcup S=L} \partial(\pi_R)b_0(w\l)\big|_{\l=\l_0} \;
\pi_S(w^{-1}x) e^{(\l_0-\rho)(x)}\\
&\qquad =\Big[ b_0(\l_0)\pi_{wL}(x)+\text{(poly in $x$, depending on $\l_0$ and $w$,
of degree $<|L|$)}\Big]
e^{(\l_0-\rho)(x)}\,.
\end{align*}
The polynomial $\pi_{wL}(x)$ has maximal degree (equal to $|\Sigma_{\l_0}^0|$) if $|J|=|\Sigma_{\l_0}^>|$.
Collecting together these terms proves then the required formula.
\end{proof}

To sum the contributions of the terms corresponding to $\mu=0$ and
$w\in W_{\l_0}$, we still need two more lemmas.

\begin{Lemma} \label{lemma:fL}
For $L \subset I$ with $|L|=|\Sigma_{\l_0}^0|$ define
\begin{equation}
 \label{eq:fL}
f_L(x)=\sum_{w \in W_{\l_0}} (\det w) \pi_{wL}(x)\,, \qquad x \in \frak a\,.
\end{equation}
Then there is a constant $c_L$ so that $f_L(x)=c_L \pi_0(x)$ for all $x \in \frak a$.
\end{Lemma}
\begin{proof}
The polynomial $f_L$ is of the same degree as $\pi_0$ and it is a $W_{\l_0}$-skew-symmetric,
so divisible by $\pi_0$. See e.g. \cite[Lemma 10]{HC-diff}
for a proof of the latter fact.
\end{proof}

\begin{Lemma} \label{lemma:constant-for-Wlambda0}
Let $c_L$ denote the constants introduced in Lemma \ref{lemma:fL} and keep the notation from
Lemma \ref{lemma:termWl0-0}.
Set
\begin{equation}
 \label{eq:rho0}
\rho_0=\sum_{\a\in \Sigma_{\l_0}^0} \a\,.
\end{equation}
Then
\begin{equation} \label{eq:constant-for-Wlambda0}
 \sum_{\substack{J\sqcup L=I\\ |J|=|\Sigma_{\l_0}^>|}} c_L \partial(\pi_J)(p_1)\big|_{\l=\l_0}
=\sum_{\substack{J\sqcup L=I\\ |J|=|\Sigma_{\l_0}^>|}} c_L
\partial(\pi_J)(\pi_1)\big|_{\l=\l_0} =\frac{1}{\pi_0(\rho_0)}\;
\partial(\pi)(\pi)\,.
\end{equation}
\end{Lemma}
\begin{proof}
The first equality in (\ref{eq:constant-for-Wlambda0}) follows
from Lemma \ref{lemma:polydiff}, (a), part 2). To prove the second
equality, notice first that, by Lemmas \ref{lemma:polydiff} and
\ref{lemma:fL}, we have
\begin{align*}
\partial(\pi)\Big[ \pi_1(\l) \big( \sum_{w \in W_{\l_0}} (\det w) e^{w\l(x)}\big) \Big]\Big|_{\l=0}&=
  \sum_{w \in W_{\l_0}} (\det w)  \partial(\pi)\big(\pi_1(\l)e^{w\l(x)}\big)\big|_{\l=0}\\
 &=\sum_{w \in W_{\l_0}} (\det w) \sum_{\substack{J\sqcup L=I\\ |J|=|\Sigma_{\l_0}^>|}} \partial(\pi_J)(\pi_1) \pi_{wL}(x)\\
 &=\sum_{\substack{J\sqcup L=I\\ |J|=|\Sigma_{\l_0}^>|}} \partial(\pi_J)(\pi_1) f_L(x)\\
 &=\pi_0(x) \; \sum_{\substack{J\sqcup L=I\\ |J|=|\Sigma_{\l_0}^>|}} c_L \, \partial(\pi_J)(\pi_1) \,.
\end{align*}
Choose $x=x_{\rho_0}$. Then
\begin{equation*}
 \sum_{w \in W_{\l_0}} (\det w) e^{w\l(x_{\rho_0})}= \sum_{w \in W_{\l_0}} (\det w) e^{w\rho_0(x_\l)}
=\prod_{\a\in \Sigma_{\l_0}^0} \sinh \inner{\a}{\l}\,.
\end{equation*}
The last equality is e.g. Proposition 5.15 (i) in \cite[Ch.
II, \S 5]{He2}. Since $\pi_0(x_{\rho_0})=\pi_0(\rho_0)> 0$, we can write
\begin{equation*}
 \sum_{\substack{J\sqcup L=I\\ |J|=|\Sigma_{\l_0}^>|}} c_L \partial(\pi_J)(\pi_1)= \frac{1}{\pi_0(\rho_0)} \;
\partial(\pi)\Big[ \pi_1(\l) \prod_{\a\in \Sigma_{\l_0}^0} \sinh \inner{\a}{\l} \Big]\Big|_{\l=0}\,.
\end{equation*}
We claim that
\begin{equation}
 \partial(\pi)\Big[ \pi_1(\l) \prod_{\a\in \Sigma_{\l_0}^0} \sinh \inner{\a}{\l} \Big]\Big|_{\l=0}=
\partial(\pi)(\pi)\,.
\end{equation}
Indeed
\begin{equation*}
 \partial(\pi)\Big[ \pi_1(\l) \prod_{\a\in \Sigma_{\l_0}^0} \sinh \inner{\a}{\l} \Big]\Big|_{\l=0}
=\sum_{\substack{R\sqcup S=I\\ |R|=|\Sigma_{\l_0}^>|}} \partial(\pi_R)(\pi_1)
\Big[ \partial(\pi_S) \Big(\prod_{\a\in \Sigma_{\l_0}^0} \sinh \inner{\a}{\l}\Big) \Big]\Big|_{\l=0}\,.
\end{equation*}
Because of the evaluation at $\l=0$, the only nonzero terms inside the last square parenthesis are those of the form
\begin{equation*}
 \prod_{j=1}^d \partial(\inner{\l}{\b_j}) (\sinh \inner{\gamma_j}{\l})\big|_{\l=0}
\end{equation*}
where $d=|\Sigma_{\l_0}^0|$, and $\{\b_1,\dots,\b_d\}$ and $\{\gamma_1,\dots,\gamma_d\}$ are respectively enumerations of $S$ and
$\Sigma_{\l_0}^0$.
The conclusion follows as
\begin{equation*}
 \partial(\inner{\l}{\b}) (\sinh \inner{\gamma}{\l})\big|_{\l=0}=\inner{\gamma}{\b} \cosh \inner{\gamma}{\l})\big|_{\l=0}=\inner{\gamma}{\b}\,,
\end{equation*}
so $\partial(\pi_S) \Big(\prod_{\a\in \Sigma_{\l_0}^0} \sinh \inner{\a}{\l}\Big) \Big|_{\l=0}=\partial(\pi_S)(\pi_0)\,.$
\end{proof}

\begin{Cor} \label{cor:sumoverWlambda0}
Keep the above notation. Then
\begin{equation}
 \sum_{w\in W_{\l_0}} \partial(\pi) \Big( p(\l)c(w\l) e^{(w\l-\rho)(x)}\Big)\Big|_{\l=\l_0}
=\Big( \frac{c_0}{\pi_0(\rho_0)} b_0(\l_0) \pi_0(x) + f_{\l_0}(x)\Big) e^{(\l_0-\rho)(x)}\,
\end{equation}
where $c_0=\partial(\pi)(\pi)$ is the positive constant of
Proposition \ref{prop:asymptoticsF-beginning} (c),
$b_0(\l_0)\neq 0$ and $f_{\l_0}(x)$ is a polynomial function of
$x$ of degree $<\deg \pi_0=|\Sigma_{\l_0}^0|$.
\end{Cor}
\begin{proof}
This is an immediate consequence of (\ref{eq:c-and-b}) and of Lemmas \ref{lemma:b0}, \ref{lemma:termWl0-0}, \ref{lemma:fL} and
\ref{lemma:constant-for-Wlambda0}.
\end{proof}

Similar (but less explicit) computations can be performed to evaluate each term of the series (\ref{eq:F-der-series}).
Notice that, by Lemma \ref{lemma:singcGammamu}, (b) and (c), the function defined by
\begin{equation}
\label{eq:bw}
b_w(\l)=\begin{cases}
\pi_0(\l)c(w\l) &\text{if $w \in W_{\Re\l_0}$}\\
\pi_0(\l)p_{w,-}(\l)c(w\l) &\text{if $w \in W\setminus W_{\Re\l_0}$}
\end{cases}
\end{equation}
is always holomorphic on a neighborhood of $\overline{U} \ni
\l_0$. It is nonzero at $\l=\l_0$ for $w \in W_{\Re\l_0}$.
However, Lemma \ref{lemma:b0}(a)  holds only when $w \in
W_{\l_0}$. As in the case of $W_{\l_0}$, each term of
(\ref{eq:F-der-series}) corresponding to a fixed $w \in W$ and
$\mu=0$ includes a polynomial factor in $x$. It is of degree $\leq
|\Sigma_{\l_0}^0|=\deg \pi_0$ if $w \in W_{\Re\l_0}$, and of
degree $\leq |\Sigma_{\l_0}^0|+|\Sigma_{\l_0}^> \cap
w(-\Sigma_0^+)|=\deg (\pi_0 p_{w,-})$ if $w \in W \setminus
W_{\Re\l_0}$. Estimates for each of the terms of
(\ref{eq:F-der-series}) can be obtained from Lemma
\ref{lemma:conv-diffHCseries}. The result of this computation is
presented in the following theorem.

\begin{Thm} \label{thm:HCseriesF}
Keep the assumptions of Proposition \ref{prop:asymptoticsF-beginning}, and let $x_0 \in \frak a^+$ be fixed.
Then for $x \in x_0 +\overline{\frak a^+}$ we have
\begin{multline}\label{eq:HCseriesF}
c_0 \varphi_{\l_0}(x)=
\Big( \frac{c_0}{\pi_0(\rho_0)} b_0(\l_0) \pi_0(x) + f_{\l_0}(x)\Big) e^{(\l_0-\rho)(x)}  \\
+ \sum_{w \in (W_{\Re\l_0} \setminus W_{\l_0}) \sqcup (W\setminus W_{\Re\l_0})}
\Big( b_w(\l_0) \pi_{w,\l_0}(x) + f_{w,\l_0}(x)\Big) e^{(w\l_0-\rho)(x)}    \\
+\sum_{\mu \in 2\Lambda \setminus \{0\}} \sum_{w \in W}
f_{w,\mu,\l_0}(x) e^{(w\l_0-\rho-\mu)(x)} \,.
\end{multline}
The first term in (\ref{eq:HCseriesF}) is as in Corollary
\ref{cor:sumoverWlambda0}. For $w \in W \setminus W_{\l_0}$, the
constant $b_w(\l_0)$ is given by evaluation of (\ref{eq:bw}) at
$\l=\l_0$. It is nonzero for $w \in W_{\Re\l_0} \setminus
W_{\l_0}$. The polynomial $\pi_{w,\l_0}(x)$ is explicitly given by
\begin{equation} \label{eq:piwlambda0}
\pi_{w,\l_0}(x)
=\begin{cases}
\displaystyle{\sum_{\substack{J\sqcup L=I\\ |J|=|\Sigma_{\l_0}^>|}}} \partial(\pi_J)(\pi_1)\; \pi_{wL}(x)
&\text{if $w \in W_{\Re\l_0} \setminus W_{\l_0}$}\\
\displaystyle{\sum_{\substack{J\sqcup L=I\\ |J|=|\Sigma_{\l_0}^> \cap w(\Sigma_0^+)|}}} \partial(\pi_J)(\pi_{w,+}) \;\pi_{wL}(x) \qquad
&\text{if $w \in W \setminus W_{\Re\l_0}$}
\end{cases}
\end{equation}
where
$\pi_1(\l)$ is as in (\ref{eq:pi1}) and
\begin{equation}
\label{eq:piwplus}
\pi_{w,+}(\l)=\prod_{\a \in \Sigma_{\l_0}^> \cap w(\Sigma_0^+)} \inner{\l}{\a}\,.
\end{equation}
Moreover, $f_{w,\l_0}(x)$ is a polynomial function of $x$ with $\deg f_{w,\l_0} <\deg \pi_{w,\l_0}$.

For $\mu \in 2\Lambda\setminus \{0\}$ and $w \in W$, $f_{w,\mu,\l_0}(x)$ is a polynomial function of
$x$ of degree $\leq |\Sigma_{\l_0}|=\deg p$.
The series on the right-hand side of (\ref{eq:HCseriesF}) converges uniformly for
$x \in x_0 + \overline{\frak a^+}$.
\end{Thm}

\begin{Rem}
The convergence of the series (\ref{eq:HCseriesF}) uses the following estimate, which is a consequence of Lemma
\ref{lemma:conv-diffHCseries}. Let $x_1 \in \frak a^+$ be fixed. Then there is a constant
$M_{x_1}>0$ so that for all $\mu \in 2\Lambda\setminus\{0\}$ and $w \in W$ we have
\begin{equation} \label{eq:est-fwmu}
|f_{w,\mu,\l_0}(x)| \leq  M_{x_1} (1+ |x|)^{|\Sigma_{\l_0}|} e^{\mu(x_1)}\,, \qquad x \in x_1 +\overline{\frak a^+}\,,
\end{equation}
This will also be needed in the following section, to determine
estimates of the series at infinity in $\frak a^+$ away from its
walls.
\end{Rem}

The following corollary gives some relevant special cases of Theorem \ref{thm:HCseriesF}.

\begin{Cor} \label{cor:HCseriesF}

Let $\l_0\in \frak a_\C^*$ with $\Re\l_0\in \overline{(\frak
a^*)^+}$. If not stated otherwise, we keep assumptions and
notation of Theorem \ref{thm:HCseriesF}.
\begin{enumerate}
\thmlist
\item
Suppose that $\l_0$ is generic, i.e. $(\l_0)_\a \notin \Z$ for all $\a \in \Sigma_0^+$.
Then for $x \in x_0 +\overline{\frak a^+}$ we have
\begin{equation}\label{eq:HCseriesF-regular}
\varphi_{\l_0}(x)= \sum_{w\in W} c(w\l_0) e^{(w\l_0-\rho)(x)}+
\sum_{\mu \in 2\Lambda \setminus \{0\}} \sum_{w \in W}
f_{w,\mu,\l_0} e^{(w\l_0-\rho-\mu)(x)} \,
\end{equation}
where $c(w\l_0)\neq 0$ for $w \in W_{\Re\l_0}$ and $f_{w,\mu,\l_0} \in \C$ for $w\in W$ and $\mu\in 2\Lambda\setminus\{0\}$.
\item
Suppose $\inner{\l_0}{\a}\neq 0$ for all $\a \in \Sigma_0^+$.
Then for $x \in x_0 +\overline{\frak a^+}$ we have
\begin{multline}\label{eq:HCseriesF-not-on-walls}
\varphi_{\l_0}(x)= \sum_{\l\in W_{\Re\l_0}} c(w\l_0) e^{(w\l_0-\rho)(x)}+\\
+\frac{1}{c_0} \sum_{w \in W \setminus W_{\Re\l_0}}
\Big( b_w(\l_0) \pi_{w,\l_0}(x) + f_{w,\l_0}(x)\Big) e^{(w\l_0-\rho)(x)}    \\
+\frac{1}{c_0} \sum_{\mu \in 2\Lambda \setminus \{0\}} \sum_{w \in W}
f_{w,\mu,\l_0}(x) e^{(w\l_0-\rho-\mu)(x)}
\end{multline}
with $c(w\l_0)\neq 0$ for all $w \in W_{\Re\l_0}$. \item Suppose
that $\Im{\l_0}$ belongs to the subspace of $\frak a^*$ supporting
the facet containing $\Re\l_0$, i.e. $\inner{\Im\l_0}{\a}=0$ for
all $\a\in \Sigma_0^+$ whenever $\inner{\Re\l_0}{\a}=0$. This
happens for instance if $\Im\l_0=0$. Then for $x \in x_0
+\overline{\frak a^+}$ we have
\begin{multline}\label{eq:HCseriesF-same-centralizer}
c_0 \varphi_{\l_0}(x)=
\Big( \frac{c_0}{\pi_0(\rho_0)} b_0(\l_0) \pi_0(x) + f_{\l_0}(x)\Big) e^{(\l_0-\rho)(x)}  \\
+ \sum_{w \in W \setminus W_{\Re\l_0}}
\Big( b_w(\l_0) \pi_{w,\l_0}(x) + f_{w,\l_0}(x)\Big) e^{(w\l_0-\rho)(x)}    \\
+\sum_{\mu \in 2\Lambda \setminus \{0\}} \sum_{w \in W}
f_{w,\mu,\l_0}(x) e^{(w\l_0-\rho-\mu)(x)} \,.
\end{multline}
\end{enumerate}
\end{Cor}
\begin{proof}
If $\l_0$ is generic, then $\Sigma_{\l_0}=\emptyset$ and all
polynomials in (\ref{eq:pi0}) to (\ref{eq:p}) reduce to the
constant polynomial $1$. In this case, $\partial(p)$ is the
identity operator and $c_0=1$. For all $w \in W$ the functions
$c(w\l)\Gamma_\mu(w\l)$ are nonsingular at $\l_0$, and
(\ref{eq:HCseriesF}) reduces to the (holomorphic extension of the)
classical Harish-Chandra expansion we started with.

If $\inner{\l_0}{\a} \neq 0$ for all $\a\in \Sigma_0^+$, then
$W_{\l_0}=\{\id\}$ and $\pi_0$ is the constant polynomial $1$.
Hence $b_0(\l_0)=c(\l_0)$ and $b_w(\l_0)=c(w\l_0)$ for $w
\in W_{\Re\l_0}$. Moreover, the polynomials $f_{\l_0}$ and
$f_{w,\l_0}$ (with $w \in W_{\Re\l_0}$) are identically zero.
Finally, for $w \in W_{\Re\l_0}$, the polynomials $\pi_{w,\l_0}$
are constants and
$\pi_{w,\l_0}(x)=\partial(\pi_1)(\pi_1)=\partial(p)(p)=c_0$.

The reduction in (c) follows as $W_{\Re\l_0} \subset W_{\Im\l_0}$. So $W_{\l_0}=W_{\Re\l_0}$.
\end{proof}

\section{Estimates of the hypergeometric functions}
\label{section:estimates}

\noindent We begin this section by examining the behavior of the
hypergeometric functions at infinity on $\frak a^+$. Our basic
tool is the exponential series expansion of $\varphi_\l(x)$ from Theorem
\ref{thm:HCseriesF}. However, as in the classical case of
spherical functions, one cannot work with this
series close to the walls of $\frak a^+$, but on regions of the
form $x_0 + \overline{\frak a^+}$ for a fixed $x_0\in \frak a^+$.
When $\l \in \frak a^*$, the remaining region in $\overline{\frak
a^+}$ can be handled by means of a subadditivity property proven
for $\varphi_\l(x)$ by Schapira in \cite{Schapira}.

We keep the notation of the previous section.
As in \cite[p. 164]{GV}, define $\beta:\frak a\to \R$ by
\begin{equation}
\label{eq:beta}
\beta(x)=\min_{\a \in \Pi} \a(x)\,,
\end{equation}
where $\Pi=\{\a_1,\dots,\a_l\}$ denotes as before the set of
simple roots in $\Sigma^+$. Let $x_0 \in \frak a ^+$ be fixed. Then
$|x| \asymp \b(x)$ as $x \to \infty$ in $x_0 + \overline{(\frak a^*)^+}$.

Recall that if $\Re\l_0 \in \overline{(\frak a^*)^+}$ and $w \in W$,
then $-w\Re\l_0+\Re\l_0=\sum_{j=1}^l r_j \a_j$ with $r_j \geq 0$.
Hence
\begin{equation}
\label{eq:est-exp-beta}
(\Re\l_0-w\Re\l_0)(x) \geq r_w \b(x)\,, \qquad x \in \frak a^+\,,
\end{equation}
where $r_w=\sum_{j=1}^l r_j \geq 0$. Moreover, $r_w>0$ if $w
\notin W_{\Re\l_0}$. We will say that $x \in \frak a^+$ tends to
infinity away from the walls if $\a(x) \to +\infty$ for all $\a
\in\Sigma_0^+$, i.e. if $\b(x)\to +\infty$.

\begin{Thm} \label{thm:leading-termF}
Let $\l_0 \in \frak a_\C^*$ with $\Re\l_0 \in \overline{(\frak a^*)^+}$, and let $x_0 \in \frak a^+$ be fixed.
Then there are constants $C_1>0$, $C_2>0$ and $b>0$ (depending on $\l_0$ and $x_0$)
so that for all $x \in x_0+ \overline{\frak a^+}$:
\begin{multline} \label{eq:restF-est}
\Big| \frac{\varphi_{\l_0}(x) e^{-(\Re\l_0-\rho)(x)}}{\pi_0(x)} -\Big( \frac{b_0(\l_0)}{\pi_0(\rho_0)} e^{i\Im\l_0(x)} +
\sum_{w\in W_{\Re\l_0} \setminus W_{\l_0}} \! \! \frac{b_w(\l_0)\pi_{w,\l_0}(x)}{c_0\pi_0(x)}
e^{i w\Im\l_0(x)} \Big)\Big|  \\
\leq C_1 (1+\b(x))^{-1} +C_2 (1+ \b(x))^{|\Sigma_{\l_0}^>|} e^{-b\beta(x)}\,.
\end{multline}
The term $C_1 (1+\b(x))^{-1}$ on the right-hand side of (\ref{eq:restF-est}) does not occur if $\Sigma_{\l_0}^0= \emptyset$.
\end{Thm}
\begin{proof}
According to Theorem \ref{thm:HCseriesF},
the left-hand-side of (\ref{eq:restF-est}) is bounded by
\begin{multline*}
\frac{|f_{\l_0}(x)|}{c_0 \pi_0(x)} +\sum_{w\in W_{\Re\l_0}\setminus W_{\l_0}} \frac{|f_{w,\l_0}(x)|}{c_0\pi_0(x)} +
\sum_{w\in W \setminus W_{\Re\l_0}} \frac{|b_w(\l_0)||\pi_{w,\l_0}(x)|+|f_{w,\l_0}(x)|}{c_0 \pi_0(x)}
e^{(w\Re\l_0-\Re\l_0)(x)} \\+ \sum_{\mu\in 2\Lambda \setminus \{0\}} \sum_{w \in W} \frac{|f_{w,\mu,\l_0}(x)|}{c_0 \pi_0(x)} e^{(w\Re\l_0-\Re\l_0-\mu)(x)}\,.
\end{multline*}
The polynomials $f_{\l_0}$ and $f_{w,\l_0}$ do not occur if $\Sigma_{\l_0}^0=\emptyset$ (as $\pi_0(x)\equiv 1$ in this case). Suppose then $\Sigma_{\l_0}^0\neq\emptyset$. Let $f$ be one of the functions $f_{\l_0}$ or $f_{w,\l_0}$ with
$w \in W_{\Re\l_0}\setminus W_{\l_0}$. Set $d_f=\deg f(x)$ and $d=\deg \pi_0(x)=|\Sigma_{\l_0}^0|$. Hence
$d_f<d$. Then for some constant $C_f>0$ we have
$|f(x)|\leq C_f (1+|x|)^{d_f}$. For $\a\in \Sigma_0^+$ and $x \in x_0 + \overline{\frak a^+}$ we also have $\a(x) \geq \b(x) \geq C' |x|$ for some constant $C'>0$ since $|x| \asymp \b(x)$.
Thus $$\frac{|f(x)|}{\pi_0(x)} \leq C'_f (1+|x|)^{d_f-d} \leq C'_f (1+|x|)^{-1}$$ for $x \in x_0+\overline{\frak a^+}$.
This leads to the first term of the right-hand side of (\ref{eq:restF-est}).
By (\ref{eq:est-exp-beta}) and since $\deg \pi_{w,\l_0}(x)=d$, there is a constant $M>0$ so that for all $w \in W \setminus W_{\Re\l_0}$ and $x \in x_0 +\overline{\frak a^+}$ we have
$$
\sum_{w\in W \setminus W_{\Re\l_0}} \frac{|b_w(\l_0)||\pi_{w,\l_0}(x)|+|f_{w,\l_0}(x)|}{c_0 \pi_0(x)}
e^{(w\Re\l_0-\Re\l_0)(x)} \leq M e^{-r\b(x)}
$$
with $r=\min_{w \in W \setminus W_{\Re\l_0}} r_w>0$.

For the last part of the estimate we proceed similarly to \cite[p. 163]{GV}.
Apply (\ref{eq:est-fwmu}) with $x_1=x_0/2$. Since $(w\Re\l_0-\Re\l_0)(x)\leq 0$, we have
\begin{align*}
\sum_{\mu\in 2\Lambda \setminus \{0\}} \sum_{w \in W} \frac{|f_{w,\mu,\l_0}(x)|}{c_0 \pi_0(x)} e^{(w\Re\l_0-\Re\l_0-\mu)(x)} &\leq M' (1+\b(x))^{|\Sigma_{\l_0}|-d} \sum_{\mu\in 2\Lambda \setminus \{0\}}
e^{-\mu(x-x_0/2)}\\
&\leq M' (1+\b(x))^{|\Sigma_{\l_0}^>|} \sum_{\mu\in \Lambda \setminus \{0\}} e^{-\mu(2x-x_0)}
\end{align*}
where $M'$ is a positive constant.
Observe that if $x \in x_0 + \overline{\frak a^+}$ then $2x -x_0 \in x_0 + \overline{\frak a^+}$.
Recall from Section \ref{subsection:HCseries} the notation $\ell(\mu)=\sum_{j=1}^l \mu_j$ for the level of
$\mu=\sum_{j=1}^l \mu_j \a_j \in \Lambda$. For every $m\in \N$ there are at most $m^{l-1}$ distinct $\mu \in \Lambda$ with $\ell(\mu)=m$. For $X \in x_0 + \overline{\frak a^+}$ we have then
\begin{equation*}
\sum_{\mu\in \Lambda \setminus \{0\}} e^{-\mu(X)} \leq \sum_{m=1}^\infty m^{l-1} e^{-m\b(X)}= e^{-\b(X)} \sum_{n=0}^\infty (n+1)^{l-1} e^{-n\b(X)}\,.
\end{equation*}
If $\b(X)\geq 1$, then
\begin{equation*}
\sum_{\mu\in \Lambda \setminus \{0\}} e^{-\mu(X)} \leq e^{-\b(X)} \sum_{n=0}^\infty (n+1)^{l-1} e^{-n}\,.
\end{equation*}
Since $\b(2x-x_0)\geq 1$ for $x \in x_0 + \overline{\frak a^+}$
and since $\b(2x-x_0)\geq 2\b(x)- \max_{j=1,\dots,l} \a_j(x_0)$,
we conclude that
\begin{equation*}
(1+\b(x))^{|\Sigma_{\l_0}^>|} \sum_{\mu\in 2\Lambda \setminus
\{0\}} e^{-\mu(x-x_0/2)} \leq C_{x_0}
(1+\b(x))^{|\Sigma_{\l_0}^>|} e^{-2\b(x)}\,.
\end{equation*}
\end{proof}

The next corollary restates Theorem \ref{thm:leading-termF} in the special case where $\l_0\in \overline{(\frak a^*)^+}$.

\begin{Cor} \label{cor:leading-termF-realcase}
Let $\l_0 \in \overline{(\frak a^*)^+}$, and let $x_0 \in \frak a^+$ be fixed. Then there are constants $C_1>0$, $C_2>0$ and $b>0$ (depending on $\l_0$ and $x_0$)
so that for all $x \in x_0+ \overline{\frak a^+}$:
\begin{multline} \label{eq:restF-est-realcase}
\Big| \varphi_{\l_0}(x) -\frac{b_0(\l_0)}{\pi_0(\rho_0)} \pi_0(x) e^{(\l_0-\rho)(x)}\Big|  \leq \\
\leq \big[ C_1 (1+\b(x))^{-1} +C_2 (1+ \b(x))^{|\Sigma_{\l_0}^>|} e^{-b\beta(x)}\big] \pi_0(x) e^{(\l_0-\rho)(x)}\,.
\end{multline}
The term $C_1 (1+\b(x))^{-1}$ on the right-hand side of (\ref{eq:restF-est-realcase}) does not occur if $\Sigma_{\l_0}^0= \emptyset$.
\end{Cor}

Let $\l\in \frak a^*$.  In \cite[Lemma 3.4]{Schapira}, Schapira proved that
for all $x \in \overline{\frak a^+}$
\begin{equation} \label{eq:localHarnack}
\nabla \varphi_{\l}(x)=-\frac{1}{|W|} \sum_{w \in W} w^{-1}(\rho-\l)G_{\l}(wx)\,,
\end{equation}
where $G_\l$ is the nonsymmetric hypergeometric function. (The gradient is taken with respect to the space variable $x\in \frak a$). Since
$\partial(\xi)F=\inner{\nabla F}{\xi}$ and because of
(\ref{eq:FandG}), one obtains for all $\xi \in \frak a$
$$
\partial(\xi) \Big( e^{K_\xi \frac{\inner{\xi}{\cdot}}{|\xi|^2}} \varphi_{\l}(\cdot) \Big) \leq 0
$$
where $K_\xi=\max_{w\in W} (\rho-\l)(w\xi)$.
This in turn yields the following subadditivity property, which is implicit in \cite{Schapira}.

\begin{Lemma}\label{lemma:subadd-Schapira}
Let $\l\in\frak a^*$. Then for all $x, x_1 \in \frak a$ we have
\begin{equation} \label{eq:subadd-Schapira}
\varphi_{\l}(x+x_1) e^{-\max_{w\in W} (\l-\rho)(wx_1)} \leq \varphi_\l(x) \leq
\varphi_{\l}(x+x_1) e^{\max_{w\in W} (\rho-\l)(wx_1)}\,.
\end{equation}
In particular, if $\l \in \overline{(\frak a^*)^+}$ and $x_1 \in \frak a^+$, then
\begin{equation} \label{eq:subadd-Schapira-special}
\varphi_{\l}(x+x_1) e^{-(\l+\rho)(x_1)} \leq \varphi_{\l}(x) \leq
\varphi_{\l}(x+x_1) e^{(\l+\rho)(x_1)}\,
\end{equation}
for all $x \in \frak a$.
\end{Lemma}

Together with Corollary \ref{cor:leading-termF-realcase}, the
above lemma yields the following global estimates of $\varphi_{\l}$, stated without
proof in \cite[Remark 3.1]{Schapira}.

\begin{Thm}\label{thm:real}
Let $\l_0 \in \overline{(\frak a^*)^+}$. Then for all $x \in \overline{\frak a^+}$ we have
\begin{equation}
\varphi_{\l_0}(x) \asymp \big[\prod_{\a\in\Sigma_{\l_0}^0} (1+\a(x))\big] e^{(\l_0-\rho)(x)}\,.
\end{equation}
\end{Thm}
\begin{proof}
We proceed as in \cite[Theorem 3.1]{Schapira} or in
\cite{AnkerCRAS}, for the case $\l_0=0$. Choose $x_1 \in x_0 +
\overline{(\frak a^*)^+}$ so that the expression inside the square
bracket at the right-end side of (\ref{eq:restF-est-realcase}) is
smaller than $\frac{b_0(\l_0)}{2\pi_0(\rho_0)}$ for $x \in x_1 +
\overline{(\frak a^*)^+}$. Applying then
(\ref{eq:subadd-Schapira-special}), we obtain the required result
from (\ref{eq:restF-est-realcase}).
\end{proof}

\section{Bounded hypergeometric functions}
\label{section:bddhyp}

Let $C(\rho)$ denote the convex hull of the points $w\rho$ ($w
\in W$).
The main result of this section is Theorem \ref{thm:bddhyp} below, which extends the
theorem of Helgason and Johnson \cite{HeJo} to the hypergeometric functions of
Heckman and Opdam. We first point out the following lemma.

\begin{Lemma} \label{lemma:Frho}
We have $\varphi_\rho \equiv 1$.
\end{Lemma}
\begin{proof}
From (\ref{eq:TG}) we have $G_{ -\rho} \equiv 1$
and so (\ref{eq:FandG}) implies that $\varphi_{ -\rho} \equiv 1.$
If $w_0$ is the longest element in $W$ then $w_0 \rho = -\rho$ and
the $W$-invariance of $\varphi_{\lambda} $ in the $\lambda$-variable
implies that $\varphi_{w\rho} \equiv 1$ for all $w \in W.$
\end{proof}

\begin{Thm} \label{thm:bddhyp}
The hypergeometric function $\varphi_\l$ is bounded if and only if
$\l$ belongs to the tube $C(\rho)+i\frak a^*$.
Moreover, $|\varphi_\lambda(x)| \leq 1$ for all $\lambda \in C(\rho) +i\frak{a}^*$ and $x\in \mathfrak a.$
\end{Thm}

\begin{proof}
First we show that $|\varphi_\l(x)| \leq 1$ if $\l \in C(\rho) + i
\frak{a}^*.$ For this, we apply maximum modulus principle to the
holomorphic function $\lambda \mapsto \varphi_\l(x)$ (for a fixed $x$).

Let $R > 0$ be arbitrary but fixed and let $B_R$ be the closed
ball of radius $R$ centered at the origin in $\frak{a}^*.$
Applying the maximum modulus principle along with $|\varphi_\l(x)|
\leq \varphi_{\Re \l}( x)$ in the domain $C(\rho) + i B_R$ implies
that the maximum of $|\varphi_\l(x)|$ is obtained when $\l$
belongs to the boundary of $C(\rho) \subset \frak{a}^*.$ Let
$\mu_1$ and $\mu_2$ be two distinct points on the boundary of
$C(\rho).$ Let $\overline{\mu_1 \mu_2}$ be the line segment in
$\frak{a}^*$ joining $\mu_1$ and $\mu_2$, and let $L_{\mu_1
\mu_2}^{\mathbb C}$ be the complex line passing through $\mu_1$
and $\mu_2,$ that is, $$ L_{\mu_1 \mu_2}^{\mathbb C} = \{ z \mu_1
+ (1-z) \mu_2:~z \in \mathbb C \}.$$ Consider the (closed) domain
$$P_\rho=(C(\rho) + i B_R) \cap \{ \lambda \in L_{\mu_1
\mu_2}^{\mathbb C}:~ \Re \lambda \in \overline{\mu_1 \mu_2} \}.$$
Now, a similar argument to the above shows that the maximum of
$\l\mapsto |\varphi_{\l}( x)|$ in $P_\rho$ is attained at either
$\mu_1$ or $\mu_2.$ It immediately follows that the maximum of
$|\varphi_\l(x)|$ in $C(\rho) + i \frak{a}^*$ is attained at the
extreme points of $C(\rho)$ which is just the set $\{ w\rho:~w \in
W \}.$ This completes the proof as $\varphi_{w\rho}(x) \equiv 1.$

We now prove that for $\lambda_0$ such that $\Re\l_0 \notin C(\rho)$
the function $\varphi_{\l_0}$ is not bounded. By $W$-invariance in the spectral
parameter, we can suppose that $\Re\l_0 \in \overline{(\frak
a^*)^+}$. We first recall that
\begin{equation}\label{eq:charCrho}
C(\rho) \cap \overline{(\frak a^*)^+} = \big(\rho-\overline{(\frak
a^*)^+}\big) \cap \overline{(\frak a^*)^+}
\end{equation}
(see \cite[Ch. IV, Lemma 8.3 (i)]{He2}).
If $\Re\l_0 \in
\overline{(\frak a^*)^+} \setminus C(\rho)$, we can therefore find $x_1
\in \frak a^+$ so that $(\Re\l_0-\rho)(x_1)>0$. Set
$d=|\Sigma_{\l_0}^0|\in \N_0$. If  $\varphi_{\l_0}$ is bounded, we have
\begin{equation} \label{eq:limitt}
\lim_{t\to +\infty} \varphi_{\l_0}(tx_1) e^{-t(\Re\l_0-\rho)(x_1)} t^{-d}=0\,.
\end{equation}
Recall that the polynomials $\pi_0$ and $\pi_{w,\l_0}$ appearing
in Theorem \ref{thm:leading-termF} are of degree $d$. Observe also
that $\pi_0(x_1)\neq 0$ as $x_1 \in \frak a^+$. Replacing $x$ by
$tx_1$ in (\ref{eq:restF-est}), we obtain as $t\to +\infty$
\begin{multline*} 
\Big| \frac{\varphi_{\l_0}(tx_1) e^{-t(\Re\l_0-\rho)(x_1)}}{t^d \pi_0(x_1)} -\Big( \frac{b_0(\l_0)}{\pi_0(\rho_0)} e^{it\Im\l_0(x_1)} \\+\sum_{w\in W_{\Re\l_0} \setminus W_{\l_0}} \frac{b_w(\l_0)\pi_{w,\l_0}(x_1)}{c_0\pi_0(x_1)}
e^{i tw\Im\l_0(x_1)} \Big)\Big| =o(t).
\end{multline*}
It follows that
\begin{equation} \label{eq:limit-contradiction}
\lim_{t\to +\infty} \Big( \frac{b_0(\l_0)}{\pi_0(\rho_0)} e^{it\Im\l_0(x_1)} +\sum_{w\in W_{\Re\l_0} \setminus W_{\l_0}} \! \! \frac{b_w(\l_0)\pi_{w,\l_0}(x_1)}{c_0\pi_0(x_1)}
e^{i tw\Im\l_0(x_1)} \Big)=0\,.
\end{equation}
We now proceed as in \cite[p. 147]{GV}.
If $u_1,\dots,u_p$ are distinct complex numbers and $c_1,\dots,c_p$ are complex constants, then
\begin{equation} \label{eq:GV}
\limsup_{t\to +\infty} \big| \sum_{j=1}^p c_j e^{iu_j t}\big|^2
\geq \lim_{T \to +\infty} \frac{1}{T}\, \int_0^T   \big| \sum_{j=1}^p c_j e^{iu_j t}\big|^2 \, dt
= \sum_{j=1}^p |c_j|^2
\end{equation}
So $\lim_{t\to +\infty} \big| \sum_{j=1}^p c_j e^{iu_j t}\big|^2=0$ implies $c_j=0$ for $j=1,\dots,p$.
Observe that if $w \in W_{\Re\l_0}\setminus W_{\l_0}$ then $w \notin W_{\Im\l_0}$. Since $x_1 \in \frak a^+$, this implies that $w\Im\l_0(x_1)\neq \Im\l_0(x_1)$ for all $w \in W_{\Re\l_0} \setminus W_{\l_0}$
Consequently, (\ref{eq:limit-contradiction}) contradicts that
$b_0(\l_0)\neq 0$. Thus $\varphi_{\l_0}$ cannot be bounded.
\end{proof}

\begin{Rem}
Some results towards Theorem \ref{thm:bddhyp} have been obtained in
\cite[Theorem 5.4 and Corollary 5.6]{Ro}. More precisely, the proof in
\cite{Ro} yields:
\begin{enumerate}
\thmlist \item $\varphi_\l$ is bounded if $\l$ belongs to the interior
of the tube $C(\rho)+i\mathfrak a^*$. (In \cite{Ro}, this is an
application of Schapira's sharp estimates proven in Theorem
\ref{thm:real} above.)
\item Suppose $\l \notin C(\rho)+i\mathfrak a^*$. If, moreover,
either $\l \in \mathfrak a^*$ or $\l \in \mathfrak a^*_\C \setminus \mathfrak
a^*$ is generic, then $\varphi_\l$ is not bounded. (This result is a
combination of Schapira's estimates on $\mathfrak a^*$ and the
classical Harish-Chandra series for generic $\l \in \fa_\C^*$.)
\end{enumerate}
Notice that (b) yields that $\varphi_\l$ is not bounded if $\l \notin C(\rho)+i\fa^*$
in the rank-one case.

Notice also that the original proof by Helgason and Johnson is based on a detailed study of
Harish-Chandra's integral formula for the spherical functions and of the boundary
components of the symmetric space $G/K$. These objects are missing in the general
theory of hypergeometric functions associated with root systems. So our proof
provides an alternative proof of the characterization of the bounded spherical functions
as well.
\end{Rem}

\section{$L^p$-Fourier analysis}

In this section we present some results towards a development of
the $L^p$-harmonic analysis for the hypergeometric Fourier
transform $\mathcal F$. We begin by considering the holomorphic
properties of $\mathcal F$ on $L^1(\fa,d\mu)^W$. This is a simple
application of Theorem \ref{thm:bddhyp}. A Riemann-Lebesgue lemma
is also obtained. We then study the properties of $\mathcal F$ on
$L^p(\fa,d\mu)^W$ with $1<p<2.$ We establish Hausdorff-Young
inequalities and as a consequence, using an argument from
\cite[pages 249--250]{FK}, we obtain injectivity and an
inversion formula for $\mathcal F$ on $L^p (\fa, d\mu)^W.$
Furthermore, we introduce the $L^p$-Schwartz space for $0<p \leq
2$ and prove an isomorphism theorem. Many of these results are
quite similar to those of the geometric case: the Hausdorff-Young
inequalities are proven as in \cite{EK82}; the $L^p$-Schwartz
space isomorphism is obtained using Anker's method
\cite{AnkerJFA}. A crucial ingredient is given by the following
estimates, due to Schapira (see \cite[Theorem 3.4 and Remark 3.2]{Schapira}): 
let $p \in S(\fa_\C^*)$ and $q \in S(\fa_\C)$ be
polynomials of degrees respectively $M$ and $N$. Then there is a
constant $C >0$ such that
\begin{equation}\label{der-estimates}
\left|\partial_\l(p) \partial_x(q)\varphi_\lambda(x)\right|\leq C
(1+|x|)^M (1+|\lambda|)^N \varphi_0(x)e^{\max_{w\in W}\Re w\lambda(x)}
\end{equation}
for all $\lambda \in \frakacs$ and $x \in \fa.$ In
(\ref{der-estimates}), we have written $\partial_y$ to indicate
that the differential operator acts on the variable $y$. These
estimates are obtained in \cite{Schapira} as a consequence of
similar estimates for the functions $G_\l$.

For $0<p\le 2$, set $\epsilon_p=\frac 2p-1$. Let
$C(\epsilon_p\rho)$ be the convex hull in $\mathfrak a^*$ of the set $\{\epsilon_p
w\rho: w\in W\}$, and let $\fa_{\epsilon_p}^*=C(\epsilon_p\rho)+ i \fa^*$.
In particular, for $p=1$, we have that $\fa_{\epsilon_1}^*=C(\rho)+i\fa^*$ is precisely the
set of parameters $\l$ for which $\varphi_\l$ is bounded.

\begin{Cor} \label{cor:holoLone}
Let $f\in L^1(\fa, d\mu)^W$. Then the following properties hold.
\begin{enumerate}
\thmlist
 \item
The hypergeometric Fourier transform $\widehat{f}(\l)$
 is well defined for all $\l \in \fa_{\epsilon_1}^*$,
and
\begin{equation} \label{eq:infty-one}
|\widehat{f}(\l)| \leq \|f\|_1\,, \qquad \l \in \fa_{\epsilon_1}^*\,.
\end{equation}
\item The function $\widehat{f}$ is continuous on
$\fa_{\epsilon_1}^*$ and holomorphic in its interior.
\item
\textup{(Riemann-Lebesgue lemma)} We have
$$ \lim_{\qquad{\l \in \fa^*_{\epsilon_1},|\Im\lambda|\rightarrow\infty}} |\widehat{f}(\lambda)|=0\,.$$
\end{enumerate}
\end{Cor}
\begin{proof}
The first two properties are immediate consequences of Theorem
\ref{thm:bddhyp}, the fact that $\varphi_\lambda$ is holomorphic in
$\l$, and Morera's theorem. The Riemann-Lebesgue lemma follows by
approximating an arbitrary function in $L^1(\frak a, d\mu)^W$ by
$W$-invariant compactly supported smooth functions and the
Paley-Wiener theorem.
\end{proof}

Next we discuss some properties of $\mathcal F$ on $L^p(\fa,d\mu)^W.$

\begin{Lemma}
Let $f\in L^p(\fa, d\mu)^W$. Then the following properties hold.
\begin{enumerate}
\thmlist
 \item
The hypergeometric Fourier transform $\widehat{f}(\l)$
 is well defined for all $\l $ in the interior of $\fa_{\epsilon_p}^*$.
It defines a holomorphic function in the interior of
$\fa_{\epsilon_p}^*$.

\item \label{eq:Hausdorff-Young} \textup{(Hausdorff-Young: real
case)} Let $p,q$ be so that $1<p<2$ and $1/p+1/q=1$. Then there is
a constant $C_p >0$ so that
$\|\widehat{f}\|_q:=\left(\int_{i\fa^*} |\widehat{f}(\l)|^q
|c(\l)|^{-2}\; d\l\right)^{1/q} \leq C_p \|f\|_p$.
\end{enumerate}
\end{Lemma}
\begin{proof}
The first part is an immediate application of the estimates (\ref{der-estimates}). The proof of the second part
can be obtained by following the methods used in \cite[ Lemma 8]{EK82}. 
More precisely, it is an application
to the Riesz-Thorin interpolation theorem to the operator $\mathcal F$, which is of type $(2,2)$ by the Plancherel
theorem and of type $(1,\infty)$ by (\ref{eq:infty-one}).
\end{proof}

The Hausdorff--Young inequality above can be extended as in
\cite{EK82}. As the proofs are very similar we only give a sketch.
Recall that $\Sigma_0$ is the set of elements in $\Sigma$ which
are not integral multiples of other elements in $\Sigma.$ For
$\alpha \in \Sigma_0$ define, $a(\alpha) = m_\alpha +
m_{2\alpha}.$
\begin{Lemma}
Let $f \in L^p(\fa, d\mu)^W, 1 < p < 2$ and $\eta$ be in the
interior of $C(\epsilon_p \rho).$ Then the following properties
hold.
\begin{enumerate}
\thmlist \item \textup{(Hausdorff-Young: complex case)} Let $p, q$
be such that $\frac{1}{p} + \frac{1}{q} = 1. $ Then there is a
constant $C_{p, \eta}$ such that for all $f \in L^p(\fa, d\mu)^W$
we have
$$ \left ( \int_{i\fa^*}|\widehat{f}(\l + \eta)|^q~|c(\l)|^{-2}
~d\l \right )^{1/q} \leq C_{p, \eta} \|f \|_p.$$

\item We have, $$\sup_{\l \in i\fa^*} |\widehat{f}(\l + \eta)| \leq
C_{p, \eta} \|f\|_p .$$

\item We have, $$ \lim_{\qquad{\l \in
\fa^*_{\epsilon_p},|\Im\lambda|\rightarrow\infty}}
|\widehat{f}(\lambda)|=0 .$$
\end{enumerate}
\end{Lemma}
\begin{proof}
To prove (a) we follow ideas from \cite{EK82}. First, we note
that, as in \cite[Lemma 5]{EK82} we have
\begin{equation}
|c(\l)|^{-2} \asymp \Pi_{\alpha \in \Sigma_0} |\langle \l,
\alpha \rangle |^2 (1 + |\langle \l, \alpha \rangle|)^{a(\alpha)
-2}\,.
\end{equation}
Next we define an
admissible family of analytic operators as follows:

For $z \in \mathbb C$ such that $ 0 \leq \textup{Re}z \leq
\frac{|\eta|}{\epsilon_p},$ let $Y_z(f)$ be the function defined
on the measure space $(i\fa^*, d\nu)$ where $d\nu(\l) =
\Pi_{\alpha \in \Sigma_0} (1 + |\langle \lambda, \alpha \rangle
|)^{a(\alpha)} d\l$ by
\begin{equation} Y_z(f) (\l) = \widehat{f}\left( z \eta/|\eta| +
\l \right )~\Pi_{\alpha \in \Sigma_0} (1 + \langle \l, \alpha
\rangle |)^{-1}~\langle z \eta/|\eta| + \l, \alpha \rangle.
\end{equation} Proceeding as in \cite{EK82}, we see that $Y_z$ is
of type $(2, 2)$ when $\textup{Re}z = 0$ and type $(1, \infty)$
when $\textup{Re}z = \frac{|\eta|}{\epsilon_p}$ with admissible
bounds. Analytic interpolation then proves (a). The next two
results are established exactly as in \cite{EK82}. We omit the
details.
\end{proof}

Now we turn to the question of injectivity and an inversion
formula for $\mathcal F.$

\begin{Thm}\label{thm:inversion}
Let $f \in L^p (\fa, d\mu)^W, 1 \leq p \leq 2.$ Then $\widehat{f}
\equiv 0$ implies that $f = 0$ almost everywhere. Moreover, if $f
\in L^p(\fa, d\mu)^W $ and $\widehat{f} \in L^1(\fa^*,
|c(\l)|^{-2}d\l)$ then
$$f(x) =
\int_{i\fa^*}~\widehat{f}(\lambda)~\varphi_{-\l}(x)~|c(\l)|^{-2}~d\l \quad
{\textup{almost everywhere}}.$$
\end{Thm}

\begin{proof}
The first part of the proof follows the argument used in \cite[Theorem 3.2]{FK}.
Fix $g \in C_c^\infty(\fa)^W.$ Consider the linear functionals
$T_g$ and $\widehat{T}_g$ on $L^p(\fa, d\mu)^W$ defined by $$T_g(h) =
\int_{\fa}~h(x) \overline{g(x)}~d\mu(x),$$ $$\widehat{T}_g(h) =
\int_{i\fa^*}~\widehat{h}(\lambda)~\overline{\widehat{g}(\lambda)}~|c(\lambda)|^{-2}~d\lambda.$$
 By the Plancherel theorem, $T_g = \widehat{T}_g$ on $L^1(\mathfrak a, d\mu)^W \cap L^2(\mathfrak a, d\mu)^W$, which is a
dense subspace of $L^p(\mathfrak a, d\mu)^W$ for $1 \leq p \leq 2.$ By H\"older's inequality we have $|T_g(h)| \leq \|h\|_p \|g\|_q$ if $\frac{1}{p} +
\frac{1}{q} = 1.$ Again, by the Hausdorff-Young inequality along
with Paley-Wiener estimates for $\widehat{g}(\lambda)$, we have
$$|\widehat{T}_g(h)| \leq \|\widehat{g}\|_p \|\widehat{h}\|_q \leq C_p \|\widehat{g}\|_p \|h\|_p.$$
Hence, both $T_g$ and $\widehat{T}_g$ are continuous linear
functionals on $L^p(\fa, d\mu)^W$ which agree on a dense subspace.
It follows that they agree on all of $L^p(\mathfrak a, d\mu)^W.$ Now, if $f \in L^p(\mathfrak a, d\mu)^W$
and $\widehat{f} \equiv 0,$ then $\widehat{T}_g(f) = T_g(f) = 0.$ So we
have,
$$\int_{\fa}f(x)~\overline{g(x)}~d\mu(x) = 0$$  for all $g \in
C_c^\infty(\fa)^W,$ which implies that $f$ vanishes almost
everywhere.

To prove the inversion formula, notice that if
$\widehat{f} \in L^1( \fa^*, |c(\l)|^{-2}d\l)^W,$ then by the Fubini's
theorem and $\overline{\varphi_\l (x)} = \varphi_{-\l}(x)$ for $\l \in
i\fa^*,$ we have
$$\widehat{T}_g(f) = \int_{\fa}~\overline{g(x)}~\left(
\int_{i\fa^*}~\widehat{f}(\lambda)
\varphi_{-\l}(x)~|c(\lambda)|^{-2}~d\lambda \right)~d\mu(x).$$
Since $T_g(f) = \widehat{T}_g(f)$ and $g$ is arbitrary we get the
result.
\end{proof}

We record the following equality obtained in the proof of Theorem \ref{thm:inversion}.

\begin{Cor}
Let $1\leq p \leq 2$. Suppose that $f \in L^p(\mathfrak a,d\mu)^W$ and $g \in C^\infty_c(\mathfrak a)^W$. Then
$$\int_{\mathfrak a} f(x)\overline{g(x)} \, dx=
\int_{i\mathfrak a^*} \widehat{f}(\l) \overline{\widehat{g}(\l)} \, |c(\l)|^{-2} \, d\l\,.$$
\end{Cor}

For $0<p\le 2$, we define the $L^p$-Schwartz space $\mathcal C^p(\mathfrak a)^W$ to be the set of
all $C^\infty$ $W$-invariant functions $f$ on $\mathfrak{a}$ such that for each $N\in\N_0$ and
$q\in \polya$,
\begin{equation}
\label{eq:Schwartzp}
\sup_{x\in\mathfrak{a}} (1+|x|)^N \varphi_0(x)^{-\frac 2p}|\partial(q)f(x)|<\infty.
\end{equation}
It is easy to check that $C_c^\infty(\mathfrak{a})^W\subset
\mathcal C^p(\mathfrak a)^W\subset L^p(\mathfrak a , d\mu)^W$.
Hence $\mathcal C^p(\mathfrak a)^W$ is dense in $L^p(\mathfrak a,d\mu)^W$.
As in the geometric case, it can be shown that $\mathcal C^p(\fa)^W$ is a
Frech\'et space with respect to the seminorms defined by the left-hand side of (\ref{eq:Schwartzp}).
It can also be shown that $\mathcal C^{p_1}(\fa)^W\subset \mathcal C^{p_2}(\fa)^W$ when
$p_1\leq p_2$ and that the inclusion map is continuous.

Let $\mathcal S(\fa_{\epsilon_p}^\ast)^W$ be the set
of all $W$-invariant functions
$F:\fa_{\epsilon_p}^\ast\rightarrow \C$ which are
holomorphic in the interior of  $\fa_{\epsilon_p}^\ast$,
continuous on $\fa_{\epsilon_p}^\ast$ and satisfy  for
all $r\in \N_0$ and $s\in\polya$
\begin{equation}
 \label{eq:Schwartzpa}
\sup_{\lambda\in \fa_{\epsilon_p}^\ast}
(1+|\lambda|)^r\left|\partial(s)F(\lambda)\right|<\infty.
\end{equation}
We note that when $p=2$, this space reduces to the usual Schwartz
space of functions on $i\fa^\ast$. It is easy to check that
$\mathcal S(\fa_{\epsilon_p}^\ast)^W$ is a Fr\'{e}chet
algebra under pointwise multiplication and with the topology
induced by the seminorms defined by the left-hand side of (\ref{eq:Schwartzpa}).
Moreover, using the Euclidean Fourier transform, one can prove that $PW(\frakacs)^W$ is a
dense subalgebra of $\mathcal S(\fa_{\epsilon_p}^\ast)^W$.

Recall that the
Euclidean Fourier transform of a sufficiently regular
function $f:\fa\rightarrow \C$
is defined by
$$\widetilde{f}(\lambda)=\int_\fa f(x)e^{\lambda(x)}\,dx\,, \qquad
\lambda\in\fa^\ast_\C.$$
It is an isomorphism between $C_c^\infty(\fa)^W$ and
$PW(\frakacs)^W$. It follows that there is an isomorphism $\mathcal A:
C_c^\infty(\fa)^W\rightarrow C_c^\infty(\fa)^W$ such that
$\widetilde{\mathcal Af}(\lambda)=\widehat{f}(\lambda)$ for all
$\lambda\in \fa^\ast_\C$. In the geometric case, $\mathcal A$ is
nothing but the Abel transform.

We now consider the Schwartz space isomorphism theorem. This
theorem was proved by Schapira \cite[Theorem 4.1]{Schapira} for
the case $p=2$ as an adaptation of Anker's method for the
geometric case (see \cite{AnkerJFA}).
Anker's proof in fact extends to the case
of the hypergeometric Fourier transform on $\mathcal C^p(\fa)$ with
$0<p\leq 2$.
We outline the main steps of the proof for the reader's convenience.

\begin{Thm}\label{Schwartz-space-iso}
For $0<p\leq 2$, the hypergeometric Fourier transform is a topological isomorphism between
$\mathcal C^p(\mathfrak a)^W$ and $\mathcal S(\fa_{\epsilon_p}^\ast)^W$.
\end{Thm}
\begin{proof}
Using (\ref{der-estimates}) and the estimates for $\varphi_0$ (which are a special case of Theorem \ref{thm:real}),
one can check, as in the geometric case (see e.g. \cite[Theorem 7.8.6]{GV}), that the hypergeometric
Fourier transform $\mathcal F$ maps $\mathcal C^p(\mathfrak a)^W$ into $\mathcal S(\fa_{\epsilon_p}^\ast)^W$
and is continuous. Since $\mathcal C^p(\mathfrak a)^W\subset \mathcal C^2(\mathfrak a)^W\subset L^2(\mathfrak a, d\mu)^W$,
it follows from the Plancherel theorem that $\mathcal F$ is
injective. We now show that $\mathcal F$ maps $\mathcal C^p(\mathfrak a)^W$ into $\mathcal
S(\fa_{\epsilon_p}^\ast)^W$ is surjective and that its inverse
map is continuous. For this, using the Paley-Wiener theorem, it is
sufficient to prove that given a seminorm $\beta$ of $\mathcal
C^p(\mathfrak a)^W$, there exists a seminorm $\eta$ of $\mathcal
S(\fa_{\epsilon_p}^\ast)^W$ and a constant $C>0$ such that $$\beta(f)\leq C \eta(\widehat{f})$$
for all $f\in
C_c^\infty(\fa)$.
For
$r\in\N_0$ and $s\in\polya$, let
$$\eta_{r,s}^{(p)}(F)=\sup_{\lambda\in \mathfrak{a_{\epsilon_p}^\ast}} (1+|\lambda|)^r |\partial(s)F(\lambda)|\,, \qquad
F\in \mathcal S(\fa_{\epsilon_p}^\ast)^W$$ and let
$$\beta(f)=\sup_{x\in\mathfrak{a}}
 (1+|x|)^N \varphi_0(x)^{-\frac 2p}|\partial(q)f(x)|\,,\qquad f\in \mathcal C^p(\mathfrak a)^W.$$
 For each positive integer $r$, let $\Gamma_r=\{x\in\fa: \rho(x^+)\leq r\}$ be the polar set of $r^{-1}\rho$
(see section \ref{subsection:hypergeomFourier} for the notation).
Then each $\Gamma_r$ is convex, compact and $W$-invariant.
Moreover,  $\fa$ is the disjoint union
 of $\Gamma_1$ and $\Gamma_r \setminus \Gamma_{r-1}$, $r \geq 2 .$

Let $f\in C_c^\infty(\fa)^W$. Then, using the inversion formula and
the estimates (\ref{der-estimates}), one checks
that
$$\sup_{x\in\Gamma_1} (1+|x|)^N \varphi_0(x)^{-\frac
2p}|\partial(q)f(x)|\leq C_1 \eta_{N_1, 1}^{(2)}(\widehat{f})$$
for some $N_1\in \N$. Here we have used that $\Gamma_1$ is
compact.

Let $w_r\in C_c^\infty(\fa)^W$ be equal to $1$ on $\Gamma_{r-1}$, equal to
$0$ on $\fa \setminus \Gamma_r$ and such that $w_r$ and all of
its derivatives are bounded, uniformly on $r \in \N$.
Since $\mathcal A$ is an isomorphism, there exists $f_r\in C_c^\infty(\fa)^W$
such that $(1-w_r)\mathcal Af=\mathcal
Af_r$. From this it follows that $f$ and $f_r$ may differ only inside $\Gamma_r$.

Using again the inversion formula on $f_r$, one observes
that $$|\partial(q)f_r(x)|\leq C_2 \varphi_0(x)\,\eta_{N_1,1}^{(2)}(\widehat{f_r}).$$
We now estimate $\eta_{N_1,1}^{(2)}(\widehat{f_r})$. Using the relation
$\widetilde{\mathcal Af_r}(\lambda)=\widehat{f_r}(\lambda)$ and  Euclidean Fourier
analysis, we prove that
\begin{equation}
 \label{eq:estim-fr}
\eta_{N_1, 1}^{(2)}(\widehat{f_r})\leq
C_3 \sum_{k=0}^{N_1}\sup_{x\in \fa\setminus\Gamma_{r-1}}(1+|x|)^{l+1}
|\nabla_{\frak a}^k \mathcal Af(x)|.
\end{equation}
Here $l=\dim \fa$ and $\nabla_{\frak a}$ is the
gradient operator on $\fa$.

Next we estimate $\sup_{x\in \Gamma_{r+1}\setminus
\Gamma_r}(1+|x|)^N \varphi_0(x)^{-\frac 2p}|\partial(q)f(x)|$. Note
that $f=f_r$ on $\Gamma_{r+1}\setminus \Gamma_r$. Then,
using the previous steps and the estimate of $\varphi_0$, we get
\begin{equation}
 \label{eq:estim-norm}
\sup_{x\in \Gamma_{r+1}\setminus \Gamma_r}(1+|x|)^N
\varphi_0(x)^{-\frac 2p}|\partial(q)f(x)|\leq C_4 r^{N_2}e^{\epsilon_p r}\eta_{N_1,1}^{(2)}(\widehat{f_r})
\end{equation}
for some $N_2\in \N$.
By (\ref{eq:estim-fr}), the right-hand side of  is dominated by
$$C_5
\sum_{k=0}^{N_1}\sup_{x\in\fa^+\setminus
\Gamma_{r-1}}(1+|x|)^{N_2+l+1} e^{\epsilon_p\rho(x)}| \nabla^k\mathcal Af(x)|.$$
In turn, this is dominated by
$$C_6
\sum_{m=0}^{N_2+l+1}\int_{i\fa^\ast}(1+|\l|)^{N_1}\left|
\nabla^m\widehat{f}\left(\l+\epsilon_p\rho\right)\right|\,d\l.$$
This inequality follows from the fact that for any polynomials $p$
and $q$
$$p(x)e^{\epsilon_p\rho(x)}\partial(q)g(x)=c\int_{i\fa^\ast}\partial(p)\left\{q\left(\lambda-\epsilon_p\rho\right)\,
h\left(-\lambda+\epsilon_p\rho\right)\right\}e^{-\lambda(x)}\,d\lambda$$
where $g(x)=\int_{i\fa^\ast}h(\lambda)e^{-\lambda(x)}\,d\lambda$.
Thus
$$\sup_{x\in \Gamma_{r+1}\setminus \Gamma_r}(1+|x|)^N \varphi_0(x)^{-\frac 2p}|\partial(q)f(x)|\leq
C_7 \sum_{m=0}^{N_2+l+1}\sup_{\lambda\in\fa_{\epsilon_p}^\ast}(1+|\lambda|)^{N_1+l+1}|\nabla^m\widehat{f}(\lambda)|.$$
This completes the proof.
\end{proof}

We conclude this section with a remark on the convolution structure of $L^1(\mathfrak a, d\mu)^W$.
Suppose that the triple $(\mathfrak a, \Sigma, m)$ is geometric and corresponds to the Riemannian
symmetric space of the noncompact type $G/K$.
The space $L^1(G/K)^K$ of $K$-invariant $L^1$ functions on $G/K$ is a commutative Banach algebra with
respect to the $L^1$-norm and a natural positive convolution structure.
This structure is transferred to $L^1(\mathfrak a, d\mu)^W$ by restriction to
$\mathfrak a\equiv \exp \mathfrak a$. The maximal ideals of $L^1(G/K)^K$ are all of the form
$$M_\l=\{f \in L^1(G/K)^K: \int_G f(g)\varphi_\l(g) \, dg=0\}$$
where $\varphi_\l$ is a bounded spherical function. It follows that, in the geometric case,
the bounded hypergeometric functions $\varphi_\l$ parametrize the characters of the commutative algebra
$L^1(\mathfrak a, d\mu)^W$.
For arbitrary triples, to find a positive convolution structure is an important open problem.
When the rank is one, this is completely settled by the results of Flensted-Jensen and Koornwinder \cite{FK}.
When the rank is greater than one, the only result available is due to R\"osler
\cite{Ro} where the root system is of the type $BC$ and multiplicity function ranges
in three continuous one-parameter families. In both cases, it has been proven that the bounded
hypergeometric functions indeed give the characters of $L^1(\mathfrak a, d\mu)^W$.

\bibliographystyle{model1b-num-names}

\end{document}